\documentclass[11pt]{amsart}

\usepackage{amsfonts}
\usepackage{amssymb}
\usepackage{graphicx}
\usepackage{newlfont}
\usepackage{mathrsfs}
\usepackage{graphics}
\usepackage{amsmath, amssymb, amsfonts,amsthm}
\usepackage{textcomp}

\usepackage[numbers,sort&compress]{natbib}

\newcommand{\bmu}{{\mathbf \mu}}

\newtheorem{theorem}{Theorem}
\newtheorem{cor}[theorem]{Corollary}
\newtheorem{lem}[theorem]{Lemma}

\newcommand{\ed}{\ \stackrel{d}{=} \ }

\newcommand{\LL}{\mbox{${\mathcal L}$}}
\newcommand{\OO}{\mbox{${\mathcal O}$}}

\newcommand{\Rbold}{\mbox{${\mathbb R}$}}

\newcommand{\Zbold}{\mbox{${\mathbb Z}$}}

\newcommand{\Ebold}{\mbox{${\mathbb E}$}}
\newcommand{\Pbold}{\mbox{${\mathbb P}$}}
\newcommand{\Ibold}{\mbox{${\mathbb I}$}}

\newcommand{\bE}{{\mathbf E}}

\begin{document}

\title{P\'olya Urn Schemes with Infinitely Many Colors}
\author{Antar Bandyopadhyay} 
\address[Antar Bandyopadhyay]{Theoretical Statistics and Mathematics Unit \\
         Indian Statistical Institute, Delhi Centre \\ 
         7 S. J. S. Sansanwal Marg \\
         New Delhi 110016 \\
         INDIA}
\address{Theoretical Statistics and Mathematics Unit, 
         Indian Statistical Institute, Kolkata;
         203 B. T. Road, Kolkata 700108, INDIA}
\email{antar@isid.ac.in}          
\author{Debleena Thacker}  
\address[Debleena Thacker]{Theoretical Statistics and Mathematics Unit \\
         Indian Statistical Institute, Delhi Centre \\ 
         7 S. J. S. Sansanwal Marg \\
         New Delhi 110016 \\
         INDIA}
\email{thackerdebleena@gmail.com}

\begin{abstract}
In this work we introduce a new type of urn model with
infinite but countable many colors indexed by an appropriate 
infinite set. We mainly consider the indexing set of colors to be the
$d$-dimensional integer lattice and consider balanced replacement schemes  
associated with bounded increment random walks on it. 
We prove central and local limit theorems for the random color of the $n$-th selected ball and 
show that irrespective of the null recurrent or transient behavior of the underlying random walks,
the asymptotic distribution is Gaussian after appropriate centering and scaling. 
We show that the order of any non-zero centering is always $\OO\left(\log n\right)$ and
the scaling is $\OO\left(\sqrt{\log n}\right)$.
The work also provides similar results for urn models with infinitely many 
colors indexed by more general lattices in $\Rbold^d$. We introduce a novel technique of
representing the random color of the $n$-th selected ball as a suitably sampled 
point on the path of the underlying random walk. This helps us to derive the central and 
local limit theorems.
\end{abstract}



\maketitle

\section{Introduction}
\label{Sec:Intro}

\subsection{Background and Motivation}
\label{SubSec:Back}
In recent years, there has been a wide variety of work on \emph{random reinforcement models} of 
various kind \cite{Da90, Vl03, Svante2, BaiHu05, FlDuPu06, 
VlTa07, Pe07, maulik1, CoVl09, DasMau11, CrGeVolWaWa11, LaPa13, CoCoLi13, ChHsYa14}. 
In particular, there has been several work on different kind of \emph{urn models}
and their generalizations 
\cite{Svante2, BaiHu05, FlDuPu06, maulik1, DasMau11, CrGeVolWaWa11, LauLimi12, 
CoCoLi13, LaPa13, ChHsYa14}. 
For \emph{occupancy urn models}, where one considers recursive addition of balls in to
finite or infinite number of boxes,
there are some works which 
introduce models with infinitely many colors, typically represented by the boxes
\cite{Du89, GneHanPit07, HwSv08}. However, other than the
classical work by Blackwell and MacQueen \cite{BlackMac73}, there has not been much development of infinite
color generalization of the P\'olya urn scheme. 
In this paper, we introduce and analyze a new P\'{o}lya type urn scheme with countably infinitely 
many colors indexed by $\Zbold^d$. 

Starting from the seminal work by P\'olya \cite{Polya30}, 
various types of urn schemes with finitely many colors have been 
widely studied in literature
\cite{Fri49, Free65, AthKar68, BagPal85, Pe90, Gouet, Svante1, Svante2,
BaiHu05, FlDuPu06, maulik1, maulik2, DasMau11, ChKu13, ChHsYa14}.
See \cite{Pe07} for an extensive survey of the known results. 
The generalized P\'olya urn scheme with finitely many colors can be described as follows: 
\begin{quote}
We start with an urn containing finitely many balls of different colors. 
At any time $n\geq 1$, a ball is selected \emph{uniformly} at random from the urn, 
the color of the selected ball is
noted, and it is returned to the urn along with a set of balls
of various colors which may depend on the color of the selected ball.
\end{quote}
The goal is to study the asymptotic properties of the configuration of the urn. 
Suppose there are $K \geq 1$ different colors and let
$U_{n}=\left(U_{n,1},U_{n,2}\ldots, U_{n,K}\right)$, where $U_{n,j}$ denotes the number of balls of color 
$j$ for $1 \leq j \leq K$. 
The dynamics of the urn model depends on the \emph{replacement policy} which can be presented by
a $K \times K$ matrix with non-negative entries, say 
$R := \left(\left(R\left(i,j\right)\right)\right)_{1 \leq i, j \leq K}$. 
In literature, $R$ 
is typically called the \emph{replacement matrix}.
The dynamics of the model can then be written as,
\begin{equation}
U_{n+1} = U_n + R_i
\label{Equ:Basic-Recurssion} 
\end{equation}
where $R_i$ is the $i$-th row of the replacement matrix $R$, where $i$ is the random 
color of the ball selected at the $\left(n+1\right)$-th draw. 
Although in the classical set up \cite{Polya30} the entries of the replacement matrix are
taken to be non-negative integers, but for studying the evaluation of the urn
such an assumption is not necessary.

A replacement matrix is said to be \emph{balanced}, if the row sums are constant. 
In this case, after every draw a constant number of balls are added to the urn. 
For such an urn, a standard technique is to 
divide each entry of the replacement matrix by the
constant row sum, thus without loss, one may assume that the row sums are all $1$. 
In that case, it is also customary to assume $U_0$ as a probability distribution on the
set of colors, which is to be interpreted as the probability distribution of 
the selected color of the first ball drawn from the urn. 
Note in this case the entries of $U_n = \left(U_{n,1},U_{n,2}\ldots, U_{n,K}\right)$
are no longer the number of balls of different colors, instead the entries of  
$U_n/\left(n+1\right)$ are the proportion of balls of various different colors. 
We will refer to it as the (random) configuration of the urn. 
It is useful noting here that the random probability mass function $U_n/\left(n+1\right)$  
represents 
the probability distribution of the random color of the $\left(n+1\right)$-th selected ball
given the $n$-th configuration of the urn. 
In other words, if $Z_n$ is the color of the ball selected at the $\left(n+1\right)$-th draw then
\begin{equation}
\mathbb{P}\left( Z_n = j \,\Big\vert\,U_0, U_1,\ldots, U_n \right) = \frac{U_{n,j}}{n+1}, \,\,\, 1 \leq j \leq K.
\label{Equ:Choise-Mass-Function} 
\end{equation}

Since $R$ is a stochastic matrix and $U_0$ a probability distribution on the set of colors,
we can now consider a Markov chain on the set of colors with transition matrix $R$ 
and initial distribution $U_0$. We call such a chain, a chain
associated with the urn model and vice-versa. In other words, given a balanced urn model we
can associate with it a unique Markov chain on the set of colors and conversely given a 
Markov chain there is an associated urn model with 
colors indexed by the state space. 
It is well known \cite{Gouet, Svante1, maulik1, maulik2, DasMau11} that the asymptotic properties of 
a balanced urn model with finitely many colors are often related to the qualitative
properties of this associated Markov chain on the finite state space.

The above formulation can now easily be generalized for infinitely many colors. More precisely, given any
set $S$ indexing the colors, a stochastic matrix $R$ on $S$ and an initial configuration $U_0$, one can define 
a process $\left(U_n\right)_{n \geq 0}$ by the equation \eqref{Equ:Basic-Recurssion} and 
\eqref{Equ:Choise-Mass-Function}. When $S$ is infinite
we will call such a process an \emph{urn model with infinitely many colors}. In this paper, we study 
such a process when $S = \Zbold^d$ and $R$ is the transition matrix of a bounded increment random walk
on $\Zbold^d$. This is a novel generalization of the P\'olya urn scheme which combines 
perhaps the two most classical models in probability theory, namely the urn model and the random walk.

Our main motivation to study such a process has been two fold. As mentioned earlier, it is known in the
literature \cite{Gouet, Svante1, maulik1, maulik2, DasMau11} 
that the asymptotic properties of a finite color urn depends on the qualitative 
properties of the under lying Markov chain. For example, for an irreducible aperiodic 
chain with $K$ colors,
it is shown in \cite{Gouet, Svante1} that
\begin{equation}
\frac{U_{n,j}}{n+1} \longrightarrow \pi_j \,\,\, \mbox{a.s.}
\label{Equ:pi-conv}
\end{equation}
for all $1 \leq j \leq K$, where $\pi=\left(\pi_j\right)_{1\leq j\leq K}$ is the unique stationary distribution. It is also know 
\cite{Svante1, Svante2}
that if the chain is reducible and $j$ is a transient state then
\begin{equation}
\frac{U_{n,j}}{n+1} \longrightarrow 0 \,\,\, \mbox{a.s.}
\label{Equ:conv-tran}
\end{equation}
Further non-trivial scalings have been derived for the reducible case \cite{Svante1, Svante2, maulik1, maulik2, DasMau11}. 
So one may conclude that asymptotic properties of
an urn model depends on the recurrence/transience of the underlying states. We want to investigate 
this relation when there are infinitely many colors. The bounded increment random walks on $\Zbold^d$ is a
rich class of examples of Markov chains on infinite states covering both the transient and null
recurrent cases. Needless to state that the no null recurrent state can appear in the finite case.
As we shall see later,
our study will indicate a significantly different phenomenon for the infinite color urn models 
associated with the bounded increment random walks on $\Zbold^d$. 
In fact, we shall show that the asymptotic configuration is approximately Gaussian,
irrespective of whether the underlying walk is transient or recurrent. 

Our other motivation comes from the work of Blackwell and MacQueen \cite{BlackMac73}, where the authors 
introduced a possibly infinite color generalization of the P\'olya urn scheme. In fact, their
generalization even allowed uncountably many colors; the set of colors 
typically taken as some Polish space. 
The model then described a process whose limiting distribution is the
\emph{Ferguson distribution} \cite{Black73, BlackMac73}, 
also known as the \emph{Dirichlet process prior} in the 
Bayesian statistics literature \cite{Fer73}. The replacement mechanism in \cite{BlackMac73} 
is a simple diagonal scheme,
which reinforces only the chosen color. As in the classical finite color P\'olya urn scheme 
where $R$ is the identity matrix, this
leads to \emph{exchangeable} sequence of colors. 
Our model complements this work where we consider 
replacement mechanisms with non-zero off diagonal entries. It is worth noting that the models we consider
do not include the Blackwell and MacQueen scheme \cite{BlackMac73} and our results show that
the asymptotic properties of our model are vastly different than those of
Blackwell and MacQueen \cite{BlackMac73}. 
We would also like to point out that due to the presence of off diagonal entries in the 
replacement matrix our models do not exhibit exchangeability and hence the techniques
used in this paper are entirely different and new.

\subsection{Model} 
\label{SubSec:Model}
Let $\left\{X_j\right\}_{j \geq 1}$ be i.i.d. random vectors taking values in $\Zbold^d$ with probability
mass function $p\left(u\right) := \Pbold\left(X_1 = u\right), u \in \Zbold^d$. We assume that the 
distribution
of $X_1$ is bounded, that is there exists a non-empty finite subset $B \subseteq \Zbold^d$ such that
$p\left(u\right) = 0$ for all $u \not\in B$. It is worthwhile to note that the assumption
of $B$ is finite may be removed. Instead, if we assume $X_1$ has 
moment generating function on an open interval around $0$, 
then all the results of this paper hold. But for simplicity, we will assume $B$ to be finite.

Throughout this paper 
we take the convention of writing all vectors as row vectors. 
Thus for a vector $x \in \Rbold^d$ we will write $x^T$ to denote it as a column vector.   
The notation $\langle \cdot , \cdot \rangle$ will denote 
the usual Euclidean inner product on $\Rbold^d$ and $\| \cdot \|$ the
the Euclidean norm. We shall always write
\begin{equation}
\begin{array}{rcl}
\bmu & := & \Ebold\left[X_1\right] \\
\varSigma & := & \Ebold\left[ X_1^T X_1 \right] \\
e\left(\lambda\right) & := & \Ebold\left[e^{\langle \lambda , X_1 \rangle}\right], \, \lambda \in \Rbold^d. \\
\end{array}
\label{Equ:Basic-Notations}
\end{equation}
We shall write
$\varSigma := \left(\left(\sigma_{ij}\right)\right)_{1 \leq i,j \leq d}$ and
assume that it is a positive definite matrix. 
Also
$\varSigma^{\frac{1}{2}}$ will denote the unique \emph{positive definite square root} of $\varSigma$,
that is, $\varSigma^{\frac{1}{2}}$ is a positive definite matrix such that
$\varSigma = \varSigma^{\frac{1}{2}} \varSigma^{\frac{1}{2}}$.
When the dimension $d=1$, we will denote the mean and variance simply by $\mu$ and $\sigma^2$ respectively
and in that case we assume $\sigma^2 > 0$.

Let $S_n := X_0 + X_1 + \cdots + X_n, n \geq 0$ be the random walk 
on $\Zbold^d$ starting at $X_0$ and with increments $\left\{X_j\right\}_{j \geq 1}$ which are independent. 
Needless to say that $\left\{S_n\right\}_{n \geq 0}$ is Markov chain with state-space $\Zbold^d$, 
initial distribution given by the distribution of $X_0$ and the transition matrix 
\[ R := \left(\left( p\left(v - u\right) \right)\right)_{u, v \in {\mathbb Z}^d}.\]

In this work, we consider 
the following infinite color generalization of P\'olya urn scheme where
the colors are indexed by $\Zbold^d$.
Let $U_n := \left(U_{n,v}\right)_{v \in {\mathbb Z}^d} \in [0, \infty)^{{\mathbb Z}^d}$
denote the configuration of the urn at time $n$, that is, 
\small
\[
\Pbold\left( \left(n+1\right)^{\mbox{th}} \mbox{\ selected ball has color\ } v 
\,\Big\vert\, U_n, U_{n-1}, \cdots, U_0 \right) 
\propto U_{n,v}, \, v \in \Zbold^d.
\]
\normalsize
Starting with $U_0$ which is a probability distribution we define $\left(U_n\right)_{n \geq 0}$
recursively as follows
\begin{equation}
\label{recurssion}
U_{n+1}=U_{n} + \chi_{n+1} R 
\end{equation}
where $\chi_{n+1} = \left(\chi_{n+1,v}\right)_{v \in {\mathbb Z}^d}$ is such that 
$\chi_{n+1,V}=1$ and $\chi_{n+1,u} = 0$ if $u \neq V$ where $V$ is the random color
chosen from the configuration $U_n$. In other words
\[
U_{n+1}=U_n + R_V
\]
where $R_V$ is the  $V^{\text{th}}$ row of the replacement matrix $R$. 
We will call 
the process $\left(U_n\right)_{n \geq 0}$ as the \emph{infinite color urn model} with
initial configuration $U_0$ and replacement matrix $R$. We will also refer to it as the 
\emph{infinite color urn model associated with the random walk $\left\{S_n\right\}_{n \geq 0}$ on $\Zbold^d$}.
Throughout this paper we will assume that 
$U_0 = \left(U_{0,v}\right)_{v \in {\mathbb Z}^d}$ is such that 
$U_{0,v} = 0$ for all but finitely many $v \in \Zbold^d$.

It is worth noting that 
\[
\sum_{u \in {\mathbb Z}^d} U_{n,u} = n + 1
\]
for all $n \geq 0$. If
$Z_n$ denotes the $\left(n+1\right)$-th selected color then
\begin{equation}
\Pbold\left(Z_n = v \,\Big\vert\, U_n, U_{n-1}, \cdots, U_0 \right) = \frac{U_{n,v}}{n+1}
\end{equation}
which implies
\begin{equation}
\Pbold\left(Z_n = v \right) = \frac{\Ebold\left[U_{n,v}\right]}{n+1}.
\end{equation}
In other words the expected proportion of the urn at time $n$ is given by the distribution of $Z_n$.

%

In Section \ref{Sec:General} we will further generalize the model when the associated random walk
takes values in other $d$-dimensional discrete lattices, for example, the \emph{triangular lattice} in
two dimensions. 

We like to note here that 
our model is a further a generalization of a subclass of models studied in \cite{CoCoLi13}, namely
the class of \emph{linearly reinforced models}. 
In \cite{CoCoLi13} the authors prove that for such models cardinality of all the colors will grow to infinity.
As we will see in the next section, our results will not only show that the cardinality of all colors will grow to infinity
but also provide the exact rates of their growths. 

\subsection{Notations}
\label{SubSec:Notations}
Most of the notations used in this paper are consistent with the literature on generalized urn models. 
For the sake of completeness we provide below a list of notations and conventions which we
use in the paper. 

\begin{itemize}
\item For two sequences $\left\{a_{n}\right\}_{n \geq 1}$ and $\left\{b_{n}\right\}_{n \geq 1}$ of positive 
real numbers, we will write $a_{n} \sim b_{n}$ if $\displaystyle \lim_{n \to \infty}\frac{a_{n}}{b_{n}}=1$.

\item As mentioned earlier, all vectors are written as row vectors unless otherwise stated. 
      For example, a finite dimensional vector 
$x \in \Rbold^d$ is written as 
$x=\left(x^{(1)},x^{(2)},\ldots,x^{(d)}\right)$ where $x^{(i)}$ denotes the $i$-th coordinate.
To be consistent with this notation matrices are multiplied to the right of the vectors.
The infinite dimensional vectors are written as $y=\left(y_{j}\right)_{j\in \mathcal{J}}$ where $y_{j}$
is the $j^{\text{th}}$ coordinate and $\mathcal{J}$ is the indexing set. Column vectors are denoted by $x^{T},$ where $x$ is a row vector. 

\item For any vector $x$, $x^2$ will denote a vector with the coordinates squared.

\item By $N_{d}\left(\bmu,\varSigma\right)$ we denote the $d$-dimensional Gaussian distribution with mean vector $\bmu \in \Rbold^d$ and
variance-covariance matrix $\varSigma$. 
For $d=1$, we simply write $N(\mu, \sigma^{2})$ with mean $\mu \in \Rbold$ and variance $\sigma^{2} > 0$.

\item The standard Gaussian measure on $\Rbold^d$ will be denoted by
$\Phi_d$ with its density by $\phi_d$ given by 
\[
\phi_d\left(x\right) := \frac{1}{\left(2 \pi \right)^{d/2}} e^{- \frac{\| x \|^2}{2}}, x \in \Rbold^d.
\]
For $d=1$, we will simply write $\Phi$ for the standard Gaussian measure on 
$\Rbold$ and $\phi$ for its density.


\item The symbol $\Rightarrow$ will denote weak convergence of probability measures.

\item The symbol $\stackrel{p}{\longrightarrow}$ will denote convergence in probability. 

\item For any two random variables/vectors $X$ and $Y$, we will write $X \stackrel{d}{=} Y$ to denote that $X$
and $Y$ have the same distribution. 

\end{itemize}

\subsection{Outline}
In the following section we state the main results, which we prove in Section \ref{Sec:Proofs}. 
In Section \ref{SubSec:Intermediate}, we state and prove two important results, which we use in the proofs 
of the main results. 
In Section \ref{Sec:General}, we further generalize our results for urns with infinitely many 
colors, where the color sets are indexed by other countable lattices on $\Rbold^d$.
In particular, we consider the example of the two dimensional triangular lattice. 
An elementary technical result which is needed in the proofs of the main results is deferred to the appendix.

\section{Main Results}
\label{Sec:Results} 
Throughout this paper we assume that
$\left(\Omega,\mathcal{F},\mathbb{P}\right)$ is a probability space on which all the random 
processes are defined.

\subsection{Weak Convergence of the Expected Configuration}
We present in this subsection the central limit theorem for the randomly selected color. The centering and scaling of the central limit theorem are of the order $\OO \left(\log n\right)$ and $\OO\left(\sqrt{\log n}\right)$ respectively. Such centering and scalings are available because the marginal distribution of the randomly selected color behaves like that of a delayed random walk, where the delay is of the order $\OO \left(\log n\right)$, see Theorem \ref{LRW1}.
\begin{theorem}
\label{GRW}
Let $\overline{\Lambda}_{n}$ be the probability measure on $\mathbb{R}^{d}$ corresponding to the probability 
vector $\frac{1}{n+1}\left(\mathbb{E}[U_{n,v}]\right)_{v \in \mathbb{Z}^{d}}$ and let
\[
\overline{\Lambda}_{n}^{cs}(A)
:= \overline{\Lambda}_{n}\left(\sqrt{\log n}A\varSigma^{-1/2}+ \bmu\log n\right),
\]
where $A$ is a Borel subset of $\Rbold^d$. 
Then, as $n \to \infty$,
\begin{equation}
\overline{\Lambda}_{n}^{cs}\Rightarrow \Phi_{d}.
\end{equation}
\end{theorem}

Recall that if $Z_n$ denotes the $\left(n+1\right)$-th selected color then
its probability mass function is given by 
$\left(\frac{\Ebold\left[U_{n,v}\right]}{n+1}\right)_{v \in {\mathbb Z}^d}$. Thus
$\overline{\Lambda}_{n}$ is the probability distribution of $Z_n$. So the following result holds trivially.
\begin{cor}
\label{Cor:GRW} 
Consider the urn model associated with the random walk $\{S_n\}_{n \geq 0}$ on $\Zbold^d \mbox{ } d \geq 1$, then 
as $n \rightarrow \infty$, 
\begin{eqnarray}\label{ED2}
\frac{Z_{n}-\bmu\log n} {\sqrt{\log n}} \Rightarrow N_{d}(0,\varSigma).
\end{eqnarray}
\end{cor}

The following result is an immediate application of the Theorem \ref{GRW}. 
\begin{cor}
Consider the urn model associated with the
simple symmetric random walk on $\mathbb{Z}^{d}, d\geq 1$. Then,
as $n\to \infty$,
\begin{eqnarray*}
 \frac{Z_{n}}{\sqrt{\log n}} \Rightarrow N_{d}(0,d^{-1}\mathbb{I}_{d}),
\end{eqnarray*}
where $\mathbb{I}_{d}$ is the $d \times d$ identity matrix. 
\end{cor}
The above result essentially shows that
irrespective of the recurrent or transient behavior of the under lying random walk, the associated
urn models have similar asymptotic behavior. In particular, the limiting distribution
is always Gaussian with universal orders for centering and scaling, namely, $\OO\left(\log n\right)$ and 
$\OO\left(\sqrt{\log n}\right)$ respectively.

\subsection{Weak Convergence of the Random Configuration}
In this subsection we will present an asymptotic result for the random configuration of the urn.
Let $\mathcal{M}_{1}$ be the space of probability measures on $\mathbb{R}^{d},\mbox{ }d\geq 1$,
endowed with the topology of weak convergence. 
Let $\Lambda_{n} \in \mathcal{M}_{1}$ be the random probability measure corresponding to the 
random probability vector $\frac{U_{n}}{n+1}$. It is easy to see that the function 
$\Lambda_n : \Omega \rightarrow {\mathcal M}_1$ is measurable. 
\begin{theorem}
\label{ASd}
Let 
\[
\Lambda^{cs}_{n}\left(A\right)=\Lambda_{n}\left(\sqrt{\log n}A\varSigma^{-1/2} 
+  \bmu \log n\right).
\]
Then, as $n \to \infty $,
\begin{equation}
\label{Eq:PrCgs}
\Lambda_{n}^{cs}\stackrel{p}{\longrightarrow} \Phi_{d} \mbox{\ in\ }\mathcal{M}_{1}.
\end{equation}
\end{theorem} 
We note that the Theorem \ref{ASd} is a stronger version of the Theorem \ref{GRW}. 

\subsection{Local Limit Theorem Type Results for the Expected Configuration}
It turns out that under certain assumptions the expected configuration of the urn at time $n$, namely, 
$\left(\frac{\Ebold\left[U_n\right]}{n+1}\right)_{n \geq 0}$ satisfies a \emph{local limit theorem}. 

\subsubsection{Local Limit Type Results for One Dimension}
\label{SubSubSec:LLT1}
In this subsection, we present the local limit theorems for urns with colors indexed by $\mathbb{Z}$.
Note that $X_1$ is a lattice random variable, so we can write 
\begin{equation}
\mathbb{P}\left(X_{1} \in a+ h\mathbb{Z}\right)=1,
\label{Equ:Span}
\end{equation}
where $a \in \mathbb{R}$ and $h>0$ is maximum value such that \eqref{Equ:Span} holds. $h$ is called the
span for $X_1$ (see Section 3.5 of \cite{Durr10}).
We define 
\begin{equation}
\mathcal{L}_{n}^{(1)} :=
\left\{x\colon x=\frac{n}{\sigma\sqrt{\log n}}a-\frac{\mu}{\sigma} \sqrt{\log n}+\frac{h}{\sigma \sqrt{\log n}}
z, \,\, z \in \mathbb{Z}\right\}.
\label{Equ:Def-L^1}
\end{equation}

\begin{theorem}\label{llt1}Assume that
$\mathbb{P}\left[X_{1}=0\right]>0$. Then, as $n \to \infty$
\begin{equation}
\sup_{x \in \mathcal{L}_{n}^{(1)}} \left\vert \sigma \frac{\sqrt{ \log n}}{h}\mathbb{P}\left(\frac{Z_{n}-
\mu \log n}{\sigma \sqrt{\log n}}=x\right)-\phi(x) \right\vert \longrightarrow 0.
\end{equation}
\end{theorem}

The above local limit theorem does not cover all cases. 
The next theorem is for the special case when the urn is associated with the
simple symmetric random walk which is not covered by Theorem \ref{llt1} or its generalization 
given in Section \ref{Sec:Proofs}.
\begin{theorem}
\label{llt4}
Assume that $\Pbold\left(X_1 = 1 \right) = \Pbold\left(X_1 = -1\right) = \frac{1}{2}$. Then, as $n \to \infty$
\begin{equation}
\sup_{x \in \mathcal{L}_{n}^{(1)}} \left\vert \sqrt{ \log n}
\mathbb{P}\left(\frac{Z_{n}}{\sqrt{\log n}}=x\right)-\phi(x) \right\vert \longrightarrow 0 
\end{equation} 
where 
$\mathcal{L}_{n}^{(1)}$ is given by \eqref{Equ:Def-L^1} with $\mu=0=a$ and $\sigma = 1 = h$. 
\end{theorem}
The following result is immediate from the above theorem.
\begin{cor}
\label{Cor:1d-0-color-LLT}
Assume that $\Pbold\left(X_1 = 1 \right) = \Pbold\left(X_1 = -1\right) = \frac{1}{2}$. Then, as $n \to \infty$
\begin{equation}
\mathbb{P}\left(Z_{n}=0\right) \sim \frac{1}{\sqrt{2 \pi \log n}}.
\end{equation} 
\end{cor}

\subsubsection{Local Limit Type Results for Higher Dimensions}
Now we consider the case $d \geq 2$. Note that 
$X_1$ is then a lattice random vector taking values in $\Zbold^d$. 
Let $\mathcal{L}$ be its \emph{minimal lattice}, that is, $\Pbold\left(X_1 \in x + \LL\right) = 1$ for 
every $x \in \Zbold^d$ such that $\Pbold\left(X_1 = x \right) > 0$ and if $\LL'$ is any closed subgroup
of $\Rbold^d$,
such that $\Pbold\left(X_1 \in y + \LL'\right) = 1$ for some $y \in \Zbold^d$, then $\LL \subseteq \LL'$
and the rank of $\LL$ is $d$. 
We refer to the pages 226 -- 227 of \cite{BhRa76} for formal definitions of the minimal lattice of a 
$d$-dimensional lattice random variable and its rank.
Let $l = \det\left(\LL\right)$ (see the pages 228 -- 229 of \cite{BhRa76} for more details).
Now let $x_0$ be such that $\Pbold\left(X_1 \in x_0 + \LL\right) = 1$ and we define
\begin{equation}
\mathcal{L}_{n}^{(d)} :=
\left\{ x\colon x = \frac{n}{\sqrt{\log n}} x_{0}\varSigma ^{-1/2}-\sqrt{\log n} \, 
\bmu \,\varSigma^{-1/2}+\frac{1}{\sqrt{\log n}} z \varSigma ^{-1/2}, \,\, z \in {\mathcal L} \right\}.
\label{Equ:Def-L^d}
\end{equation}

\begin{theorem}\label{llt2}Assume that
$\mathbb{P}\left[X_{1}=0\right]>0$. Then, as $n \to \infty$
\begin{equation}
\sup_{x \in \mathcal{L}_{n}^{(d)}} \left\vert 
\frac{\text{det}(\varSigma^{1/2})\left(\sqrt{\log n} \right)^{d}}{l}
\mathbb{P}\left(\frac{Z_{n}-\bmu\log n}{\sqrt{\log n}}\varSigma ^{-1/2}=x\right)-
\phi_{d}(x) \right\vert \longrightarrow 0. 
\end{equation} 
\end{theorem}
Observe that as in the one dimensional case the above theorem does not cover all the cases.  
The next theorem is for the special case when the urn is associated with the
simple symmetric random walk on $\Zbold^d, \mbox{ } d \geq 2,$ which is not covered by Theorem \ref{llt2}. 
\begin{theorem}\label{llt5}
Assume that $\Pbold\left(X_1 = \pm e_i\right) = \frac{1}{2d}$ for $1 \leq i \leq d$, where $e_i$ is the
$i$-th unit vector in direction $i$. Then, as $n \rightarrow \infty$ 
\begin{eqnarray}\label{lltSSRW}
\sup_{x \in \mathcal{L}_{n}^{(d)}} \left\vert \left(d\right)^{\frac{d}{2}} \left(\sqrt{\log n} \right)^{d}
\mathbb{P}\left(\frac{\sqrt{d}}{\sqrt{\log n}} Z_n =x\right)-
\phi_{d}(x) \right\vert \longrightarrow 0,
\end{eqnarray} 
where $\mathcal{L}_{n}^{(d)}$ is as defined in \eqref{Equ:Def-L^d} with $\mu = 0 = x_0$, $\varSigma = \Ibold_d$
and $\LL = \sqrt{d} \, \Zbold^d$. 
\end{theorem}
Similar to the one dimensional case, the next result is immediate from the above theorem.
\begin{cor}
\label{Cor:d-0-color-LLT}
Assume that $\Pbold\left(X_1 = \pm e_i\right) = \frac{1}{2d}$ for $1 \leq i \leq d$, where $e_i$ is the
$i$-th unit vector in direction $i$. Then, as $n \rightarrow \infty$ 
\begin{equation}
\mathbb{P}\left(Z_{n}=0\right) \sim \frac{1}{\left(\sqrt{2 \pi d \log n}\right)^d}.
\end{equation} 
\end{cor}

\noindent
{\bf Remark:} The 
assumption $\mathbb{P}\left[X_{1}=0 \right]>0$ can be removed, 
at least for some cases. Theorem \ref{llt4} and Theorem \ref{llt5} are such examples. 
Because of certain technical difficulties,
we do not know the full generality under which the local limit 
theorem holds, though we conjecture that it holds for all the cases. 

\subsection{Sketch of the Main Tools Used in the Proofs}
\label{SubSec:Tech}
There are few standard methods for analyzing finite color urn models which are mainly based on 
martingale techniques \cite{Gouet, maulik1, maulik2, DasMau11} 
and embedding into continuous time pure birth processes \cite{AthKar68, Svante1, Svante2, BaiHu05}. 
Typically the analysis of a finite color urn is heavily dependent on the 
\emph{Perron-Frobenius theory} \cite{Sene06} of matrices with positive entries 
\cite{AthKar68, Gouet, Svante1, Svante2, BaiHu05, maulik1, DasMau11}. 
The absence of such a theory
for infinite dimensional matrices makes the analysis of urn with infinitely many 
colors quite difficult and challenging. 

Our approach is to relate the $n$-th configuration of the urn to the underlying 
Markov chain, which in our case is a bounded increment random walk. In particular, we show that 
the distribution of $Z_n$, the color of the $\left(n+1\right)$-th selected ball can be 
represented by
\begin{equation}
Z_n \ed Z_{0}+\displaystyle\sum_{j=1}^{n}I_{j}X_{j},
\label{Equ:Representation-1}
\end{equation}
where $\{I_{j}\}_{j\geq 1}$ are independent Bernoulli random variables with
$\bE\left[I_{j}\right] = \frac{1}{j+1}$,
$j \geq 1$ and are independent of 
$\{X_{j}\}_{j\geq 1}$; and $Z_{0}$ is a random vector taking values in $\mathbb{Z}^{d}$ 
distributed according to the probability vector $U_{0}$ and is 
independent of $\left(\{I_{j}\}_{j\geq 1}; \{X_{j}\}_{j\geq1}\right)$.
Thus we can write
\begin{equation}
Z_n \ed S_{\tau_n},
\label{Equ:Representation-2}
\end{equation}
where $\left\{S_n\right\}_{n \geq 0}$ is the random walk with i.i.d. increments $\{X_{j}\}_{j\geq 1}$ 
starting at $X_0$
and
$\tau_n := \sum_{j=1}^n I_j$ is a stopping time which is independent of $\left(S_n\right)_{n \geq 0}$. 
This helps us to derive the central and local limit theorems which are stated earlier.
This approach of coupling with underlying Markov chain
is entirely new and it helps us to completely bypass the technical difficulties which one may face in using
the eigenvalue techniques in the infinite color case.
We present this representation as an independent result in the following section
(see Theorem \ref{LRW1}).


\section{Auxiliary Results}
\label{SubSec:Intermediate}
In this section, we present two results which we need to prove our main results.
These results are two very important tools for studying infinite color urn models associated with
random walks on $\Zbold^d$ and hence presented separately.
 
Define $\Pi_{n}\left(z\right)=\displaystyle\prod_{j=1}^{n}\left(1+\frac{z}{j}\right)$ for $z \in \mathbb{C}.$ 
It is known from Euler product formula for gamma function, which is also referred to as 
Gauss's formula (see page 178 of \cite{Con78}), that 
\begin{eqnarray}\label{Euler}
\displaystyle \lim_{n\to \infty} \frac{\Pi_{n}(z)}{n^{z}}\Gamma(z+1)=1
\end{eqnarray} uniformly on compact subsets of 
$\mathbb{C}\setminus\{-1,-2, -3, \ldots\}$.

Recall  
$e\left(\lambda\right) := \sum _{v \in B}e^{\langle\lambda, v\rangle}p(v)$ 
is the moment generating function of $X_1$. It is easy 
to note that $e\left(\lambda\right)$ is an eigenvalue of $R$ corresponding to the right eigenvector 
$x\left(\lambda\right)=\left(e^{\langle \lambda, v\rangle}\right)_{ v \in \mathbb{Z}^{d}}^{T}$. 
Let $\mathcal{F}_{n}=\sigma \left(U_{j}\colon 0\leq j\leq n\right), n \geq 0$ be the natural filtration. 
Define 
\[\overline{M}_{n}\left(\lambda\right)=\frac{U_{n}x\left(\lambda\right)}{\Pi_{n}\left(e\left(\lambda\right)\right)}\]
From the fundamental recursion (\ref{recurssion}) we get,
\[ U_{n+1}x\left(\lambda\right)=U_{n}x\left(\lambda\right)+\mathcal{X}_{n+1}Rx\left(\lambda\right)\]
Thus,\begin{eqnarray*}
\mathbb{E}\left[U_{n+1}x\left(\lambda\right)\Big{\lvert} \mathcal{F}_{n} \right]&= U_{n}x\left(\lambda\right)+e\left(\lambda\right)
 \mathbb{E}\left[\mathcal{X}_{n+1}x\left(\lambda\right)\Big{\lvert} \mathcal{F}_{n}\right]
=\left(1+\frac{e\left(\lambda\right)}{n+1}\right)U_{n}x\left(\lambda\right).
\end{eqnarray*} 
Therefore, $\overline{M}_{n}\left(\lambda\right)$ is a non-negative martingale for every 
$\lambda \in \mathbb{R}^{d}$. In particular 
$\mathbb{E}\left[\overline{M}_{n}\left(\lambda\right)\right]= \overline{M}_{0}\left(\lambda\right)$.

We now present a representation of the marginal distribution of 
$Z_{n}$ in terms of the increments $\left(X_{j}\right)_{j \geq 1}$. As mentioned earlier, this particular representation is interesting and non-trivial, as it necessarily demonstrates that the marginal distribution of the randomly selected color behaves like a delayed random walk.
\begin{theorem}
\label{LRW1}
For each $n\geq 1$,
\begin{eqnarray}\label{Marginal}
Z_{n}\stackrel{d}{=}Z_{0}+\displaystyle\sum_{j=1}^{n}I_{j}X_{j}.
\end{eqnarray} 
where $\{I_{j}\}_{j\geq 1}$ are independent Bernoulli random variables such that
$\bE\left[I_{j}\right] = \frac{1}{j+1}$,
$j \geq 1$
and are independent of 
$\{X_{j}\}_{j\geq 1}$; and $Z_{0}$ is a random vector taking values in $\mathbb{Z}^{d}$ 
distributed according to the probability vector $U_{0}$ and is 
independent of $\left(\{I_{j}\}_{j\geq 1}; \{X_{j}\}_{j\geq1}\right)$.
\end{theorem}
\begin{proof}
As noted before, the probability mass function for the color of the $\left(n+1\right)$-th
selected ball, namely $Z_n$, is 
$\left(\frac{\Ebold\left[U_{n,v}\right]}{n+1}\right)_{v in {\mathbb Z}^d}$. So 
for $\lambda \in \mathbb{R}^{d}$, the moment generating function of $Z_{n}$ is given by 
\begin{eqnarray}
      \frac{1}{n+1}\sum_{v\in \mathbb{Z}^{d}}e^{\langle\lambda,v\rangle }\mathbb{E}\left[U_{n,v}\right]
& = & \frac{\Pi_{n}\left(e(\lambda)\right)}{n+1}\mathbb{E}\left[\overline{M}_{n}(\lambda)\right] \nonumber\\
& = & \frac{\Pi_{n}\left(e(\lambda)\right)}{n+1}\overline {M}_{0}(\lambda) \nonumber\\
& = & \overline {M}_{0}(\lambda) \prod_{j=1}^{n}\left(1-\frac{1}{j+1}+\frac{e(\lambda)}{j+1}\right) \label{MGF}.
\end{eqnarray} 
The equation \eqref{Marginal} follows from \eqref{MGF}.
\end{proof}

Our next theorem states that around a non-trivial closed neighborhood of $0$ the martingales
$\left( \overline{M}_{n}\left(\lambda\right) \right)_{n \geq 0}$ are uniformly (in $\lambda$)
${\mathcal L}_2$ bounded. 
\begin{theorem}
\label{Martingale}
There exists $\delta > 0$ such that 
\begin{equation}
\sup_{\lambda \in \left[-\delta, \delta\right]^{d}} \sup_{n\geq 1} 
\mathbb{E} \left[\overline{M}^{2}_{n}\left(\lambda\right)\right]<\infty.
\label{Equ:L-2-bound}
\end{equation}
\end{theorem}
 
\begin{proof} 
From (\ref{recurssion}), we obtain
\begin{eqnarray*}
\mathbb{E}\left[\left(U_{n+1}x\left(\lambda\right)\right)^{2}\Big{\lvert }\mathcal{F}_{n}\right]
& = & \left(U_{n}x\left(\lambda\right)\right)^{2}
      +2e\left(\lambda\right)U_{n}x\left(\lambda\right)
      \mathbb{E}\left[\mathcal{X}_{n+1}x\left(\lambda\right)\Big{\lvert }\mathcal{F}_{n}\right] \\
&   & \quad  + e^{2}\left(\lambda\right)
      \mathbb{E}\left[\left(\mathcal{X}_{n+1}x\left(\lambda\right)\right)^{2}\Big{\lvert }\mathcal{F}_{n}\right]
\end{eqnarray*}
It is easy to see that 
\begin{equation}
\mathbb{E}\left[\mathcal{X}_{n+1}x\left(\lambda\right)\Big{\lvert} \mathcal{F}_{n}\right]=\frac{1}{n+1}U_{n}x\left(\lambda\right)\text{ and }
\mathbb{E}\left[\left(\mathcal{X}_{n+1}x\left(\lambda\right)\right)^{2}\Big{\lvert }\mathcal{F}_{n}\right]=\frac{1}{n+1}U_{n}x\left(2\lambda\right).
\end{equation}
Therefore, we get the recursion
\begin{eqnarray}
\mathbb{E}\left[\left(U_{n+1}x\left(\lambda\right)\right)^{2}\right]=\left(1+\frac{2e\left(\lambda\right)}{n+1}\right)
\mathbb{E}\left[\left(U_{n}x\left(\lambda\right)\right)^{2}\right] \nonumber \\ 
 \quad +\frac{e^{2}\left(\lambda\right)}{n+1}\mathbb{E}\left[U_{n}x\left(2\lambda\right)\right]. \label{2M}
\end{eqnarray}
Dividing both sides of (\ref{2M}) by $\Pi^{2}_{n+1}\left(\lambda\right)$,
\begin{eqnarray}\label{22M}
 \mathbb{E}\left[\overline{M}^{2}_{n+1}\left(\lambda\right)\right]=\frac{\left(1+\frac{2e\left(\lambda\right)}{n+1}\right)}
{\left(1+\frac{e\left(\lambda\right)}{n+1}\right)^{2}}\mathbb{E}\left[\overline{M}^{2}_{n}\left(\lambda\right)\right]+\frac
{e^{2}\left(\lambda\right)}{n+1}\frac{\mathbb{E}\left[U_{n}x\left(2\lambda\right)\right]}{\Pi^{2}_{n+1}\left(\lambda\right)}.
\end{eqnarray}

$\overline{M}_{n}\left(2\lambda\right)$ being a martingale, we obtain $\mathbb{E}\left[U_{n}x\left(2\lambda\right)\right]=
\Pi_{n}\left(e\left(2\lambda\right)\right)\overline{M}_{0}\left(2\lambda\right)$. Therefore from (\ref{22M}), we get
\begin{eqnarray}
      \mathbb{E}\left[\overline{M}_{n}^{2}\left(\lambda\right)\right]
& = & \frac{\Pi_{n}\left(2e\left(\lambda\right)\right)}{\Pi_{n}\left(e\left(\lambda\right)\right)^{2}}
      \overline{M}_{0}^{2}\left(\lambda\right) \nonumber \\
&   & \quad + \sum_{k=1}^{n} \frac{e^{2}\left(\lambda\right)}{k}
      \left\{\prod_{j>k}^{n}\frac{\left(1+\frac{2e\left(\lambda\right)}{j}\right)}
                                 {\left(1+\frac{e\left(\lambda\right)}{j}\right)^{2}}\right\}
      \frac{\Pi_{k-1}\left(e\left(2\lambda\right)\right)}
           {\Pi_{k}^{2}\left(e\left(\lambda\right)\right)}\overline{M}_{0}\left(2\lambda\right). \label{M2}
\end{eqnarray}
We observe that as $e\left(\lambda\right)>0$, so
$\frac{1+\frac{2e\left(\lambda\right)}{j}}{\left(1+\frac{e\left(\lambda\right)}{j}\right)^{2}} \leq 1$ 
and hence 
$\frac{\Pi_{n}\left(2e\left(\lambda\right)\right)}{\Pi^{2}_{n}\left(e\left(\lambda\right)\right)}\leq 1$.
Thus
\begin{eqnarray}\label{B1}
 \mathbb{E}\left[\overline{M}^{2}_{n}\left(\lambda\right)\right]\leq \overline{M}^{2}_{0}\left(\lambda\right)+e^{2}\left(\lambda\right)
\overline{M}_{0}\left(2\lambda\right)\displaystyle \sum_{k=1}^{n}\frac{1}{k}\frac{\Pi_{k-1}\left(e\left(2\lambda\right)\right)}
{\Pi^{2}_{k}\left(e\left(\lambda\right)\right)}\mbox{.}
\end{eqnarray}
Using (\ref{Euler}), we know that 
\begin{equation}
\Pi_{n}^{2}\left(e\left(\lambda\right)\right)\sim \frac{n^{2e\left(\lambda\right)}}{\Gamma^{2}\left(e\left(\lambda\right)+1\right)}.
\end{equation}
Since $e\left(0\right)=1$ and $e\left(\lambda\right)$ is continuous as a function of $\lambda$, 
so given $\eta>0$, there exists 
$0<K_{1},K_{2}<\infty$, such that for all $\lambda \in \left[-\eta, \eta\right]^{d}$, $K_{1}\leq e\left(\lambda\right)\leq K_{2}$. Since the convergence in (\ref{Euler}) is uniform on compact subsets of $\left[0,\infty\right),\mbox{ given }\epsilon>0$ there exists $N_{1}>0$ such that for all $n\geq N_{1}$ and $\lambda \in \left[-\eta, \eta\right]^{d}$,
\begin{eqnarray*}
&     &  \left(1-\epsilon\right)\frac{\Gamma^{2}\left(e\left(\lambda\right)+1\right)}
                             {\Gamma \left(e\left(2\lambda\right)+1\right)}
\sum_{k\geq N_{1}}^{n} \frac{1}{k^{1+2e\left(\lambda\right)-e\left(2\lambda\right)}} \\
& \leq & \sum _{k\geq N_{1}}^{n}\frac{1}{k}\frac{\Pi_{k-1}\left(e\left(2\lambda\right)\right)} 
                                {\Pi^{2}_{k}\left(e\left(\lambda\right)\right)} \\
& \leq &\left(1+\epsilon\right)\frac{\Gamma^{2}\left(e\left(\lambda\right)+1\right)}{\Gamma \left(e\left(2\lambda\right)+1\right)}
\displaystyle \sum_{k\geq N_{1}}^{n}\frac{1}{k^{1+2e\left(\lambda\right)-e\left(2\lambda\right)}}.
\end{eqnarray*}
Recall that $e\left(\lambda\right)=\textstyle\sum _{v \in B}e^{\langle\lambda, v\rangle}p(v)$. 
Since the cardinality of $B$ is finite, we can choose a $\delta_{0}>0 $ such that for every $\lambda \in \left[-\delta_{0}, \delta_{0}\right]^{d}$, 
$2e\left(\lambda\right)-e\left(2\lambda\right)>0$. Choose $\delta =\textstyle \min \{\eta,\delta_{0}\}$.
Since $2e\left(\lambda\right)-e\left(2\lambda\right)$ is continuous as a function of $\lambda$, there exists a $\lambda_{0}\in \left[-\delta,\delta\right]^{d}$ such that 
$\textstyle\min_{\lambda \in \left[-\delta,\delta\right]^{d}}2e\left(\lambda\right)-e\left(2\lambda\right)
=2e\left(\lambda_{0}\right)-e\left(2\lambda_{0}\right)>0$. Therefore 
\begin{eqnarray*}
 \displaystyle \sum_{k=1}^{\infty}\frac{1}{k^{1+2e\left(\lambda\right)-e\left(2\lambda\right)}}\leq 
 \sum_{k=1}^{\infty}\frac{1}{k^{1+2e\left(\lambda_{0}\right)-e\left(2\lambda_{0}\right)}}.
\end{eqnarray*}
Therefore given $\epsilon >0$
there exists $N_{2}>0 $ such that $\forall \lambda \in \left[-\delta, \delta\right]^{d}$. 

\begin{eqnarray*}
 \displaystyle \sum_{k>N_{2}}^{\infty}\frac{1}{k^{1+2e\left(\lambda\right)-e\left(2\lambda\right)}}\leq \sum_{k>N_{2}}^{\infty}\frac{1}
{k^{1+2e\left(\lambda_{0}\right)-e\left(2\lambda_{0}\right)}}<\epsilon.
\end{eqnarray*}
$\textstyle\frac{\Gamma^{2}\left(e\left(\lambda\right)+1\right)}
{\Gamma \left(e\left(2\lambda\right)+1\right)}$, $e^{2}\left(\lambda\right)$ and $\overline{M}_{0}\left
(2\lambda\right)$ being continuous as functions of $\lambda$ are bounded for $ \lambda \in \left[-\delta, \delta\right]^{d}$. 
Choose $N=\max\{N_{1},N_{2}\}$. From (\ref{B1}) we obtain for all $n\geq N$ 
\begin{eqnarray}\label{Delta1}
 \mathbb{E}\left[\overline{M}^{2}_{n}\left(\lambda\right)\right]\leq \overline{M}^{2}_{0}\left(\lambda\right)+C_{1}
\displaystyle \sum_{k=1}^{N}\frac{1}{k}\frac{\Pi_{k-1}\left(e\left(2\lambda\right)\right)}
{\Pi^{2}_{k}\left(e\left(\lambda\right)\right)}+\epsilon
\end{eqnarray} for an appropriate positive constant $C_{1}$.

$\textstyle\sum_{k=1}^{N}\frac{1}{k}\frac{\Pi_{k-1}\left(e\left(2\lambda\right)\right)}
{\Pi^{2}_{k}\left(e\left(\lambda\right)\right)}$ and $\overline{M}^{2}_{0}\left(\lambda\right)$ being continuous as functions of $\lambda$, are bounded for $\lambda \in \left[-\delta,\delta\right]^{d}$.
Therefore, from (\ref{Delta1}) we obtain that there 
exists $C>0$ such that
for all $\lambda \in \left[-\delta,\delta\right]^{d}$ and for all $n\geq 1$
\begin{eqnarray*}
 \mathbb{E}\left[\overline{M}^{2}_{n}\left(\lambda\right)\right]\leq C.
\end{eqnarray*}
This proves \eqref{Equ:L-2-bound}.
\end{proof}

\section{Proofs of the Main Results}\label{Sec:Proofs}
We first note that to derive the central and local limit theorems, without loss, we may assume that 
the initial configuration of the urn consists of one ball of color $0$, that is, 
$Z_0 \equiv 0$. Hence, it follows from \eqref{Marginal} that 
\begin{equation*}
Z_{n}\stackrel{d}{=}\displaystyle \sum_{j=1}^{n}I_{j}X_{j}.
\end{equation*}

\subsection{Proofs for the Expected Configuration}
\begin{proof}[Proof of Theorem \ref{GRW}]
Observe that 
\begin{equation}
\mathbb{E}\left[\sum_{j=1}^{n}I_{j}X_{j}\right]-\bmu \log n 
= \sum_{j=1}^{n}\frac{1}{j}\bmu- \bmu\log n\longrightarrow \gamma\bmu, 
\end{equation}
where $\gamma$ is the Euler's constant. \\

\noindent
{\bf Case I:} Let $d=1$. Let $s^{2}_{n}=\mbox{Var}\left(\sum_{j=1}^{n}I_{j}X_{j}\right)$.
It is easy to note that 
\[
s^{2}_{n} = \sum_{j=1}^{n}\frac{1}{j+1}\mathbb{E}\left[X_{1}^{2}\right]-\frac{\mu^{2}}{(j+1)^{2}}
\sim \sigma ^{2}\log n. 
\]
As the cardinality of $B$ is finite, so for any $\epsilon>0$, we have 
\[
\frac{1}{s^{2}_{n}} \sum_{j=1}^{n}
\mathbb{E}\left[I_{j}X^{2}_{j}1_{\{I_{j}X_{j}>\epsilon s_{n}\}}\right]\longrightarrow 0 
\]
as $n \to \infty$. 
Therefore, by the Lindeberg Central Limit theorem, we conclude that as $n \to \infty$
\[
\frac{Z_{n}-\mu \log n}{\sigma \sqrt{\log n}}\Rightarrow N(0,1).
\]
This completes the proof in this case.\\

\noindent
{\bf Case II:} Now suppose $d\geq 2$.
Let $\varSigma_{n}=\left[\sigma_{k,l}(n)\right]_{d\times d
}$ denote the variance-covariance matrix for $\textstyle\sum_{j=1}^{n}I_{j}X_{j}$. Then by calculations
similar to that in one-dimension it is easy to see that for all $k,l \in \{1,2,\ldots d\}$ as $n \to \infty$ 
\begin{eqnarray*}
 \frac{\sigma_{k,l}(n)}{(\log n)\sigma_{k,l}}\longrightarrow 1. 
\end{eqnarray*}
Therefore for every $\theta \in \mathbb{R}^{d}$, by Lindeberg Central Limit Theorem in one dimension, 
\begin{eqnarray*}
 \frac{\langle \theta,\displaystyle \sum_{j=1}^{n}I_{j}X_{j}\rangle -\langle \theta, \bmu\log n\rangle }
{\sqrt{\log n}\left(\theta\varSigma \theta^T\right)^{1/2} }\Rightarrow N(0,1)\mbox{ as }n \to\infty.
\end{eqnarray*}
Therefore by Cramer-Wold device, it follows that as $n \to \infty$
\begin{eqnarray*}
 \frac{\displaystyle \sum_{j=1}^{n}I_{j}X_{j}-\bmu\log n}{\sqrt{\log n}}\Rightarrow N_{d}\left(0,\varSigma\right).
\end{eqnarray*} 
So we conclude that as $n \to \infty$
\begin{eqnarray*}
 \frac{Z_{n}- \bmu \log n}{\sqrt{\log n}} \Rightarrow N_{d}\left(0,\varSigma\right).
\end{eqnarray*}
This completes the proof.
\end{proof}

\subsection{Proofs for Random Configuration}
In this subsection we will present the proof of Theorem \ref{ASd}. 
We start with the following lemma which is needed in the proof of Theorem \ref{ASd}.
\begin{lem}\label{PC}
Let $\delta$ be as in Theorem \ref{Martingale}, then for every $\lambda \in \left[-\delta, \delta \right]^{d}$ 
as $n \to \infty$,
\begin{equation}
\overline{M}_{n}\left(\frac{\lambda}{\sqrt{\log n}}\right)\stackrel{p}{\longrightarrow} 1.
\end{equation}
\end{lem}

\begin{proof}
From equation (\ref{M2})
we get 
\begin{eqnarray*}
 \mathbb{E}\left[\overline{M}^{2}_{n}\left(\lambda\right)\right]=\frac{\Pi_{n}\left(2e(\lambda)\right)}{\Pi^{2}_{n}\left(e(\lambda)\right)
}+\frac{\Pi_{n}\left(2e(\lambda)\right)}{\Pi^{2}_{n}\left(e(\lambda)\right)}\displaystyle\sum_{k=1}^{n}\frac{e^{2}(\lambda)}{k}
\frac{\Pi_{k-1}\left(e(2\lambda)\right)}{\Pi_{k}\left(2e(\lambda)\right)}.
\end{eqnarray*}
Replacing $\lambda$ by $\lambda_{n}=\frac{\lambda}{\sqrt{\log n}}$, we obtain
\begin{eqnarray}\label{Eq:Martingale} 
\nonumber \mathbb{E}\left[\overline{M}^{2}_{n}\left(\lambda_{n}\right)\right]=\frac{\Pi_{n}\left(2e\left(\lambda_{n}\right)\right)}
{\Pi^{2}_{n}\left(e\left(\lambda_{n}\right)\right)}+ \frac{\Pi_{n}\left(2e\left(\lambda_{n}\right)\right)}
{\Pi^{2}_{n}\left(e\left(\lambda_{n}\right)\right)} \displaystyle \sum_{k=1}^{n}\frac{e^{2}\left(\lambda_{n}\right)}{k}
\frac{\Pi_{k-1}\left(e\left(2\lambda_{n}\right)\right)}{\Pi_{k}\left(2e\left(\lambda_{n}\right)\right)}\\
\mbox{}
\end{eqnarray}
Since the convergence in formula (\ref{Euler}) is uniform on compact sets of $\left[0,\infty \right)$, we observe that for $\lambda \in \left[-\delta, \delta\right]^{d}$
\begin{eqnarray*}
 \displaystyle \lim_{n \to \infty}\frac{\Pi_{n}\left(2e\left(\lambda_{n}\right)\right)}{\Pi^{2}_{n}\left(e\left(\lambda_{n}\right)\right)}
=\frac{\Gamma^{2}\left(2\right)}{\Gamma\left(3\right)}=\frac{1}{2}.
\end{eqnarray*}
We observe that $\textstyle\lim_{n \to \infty }e\left(\lambda_{n}\right)=1$ and 
\[ 
\lim_{n \to \infty}\frac{\Pi_{n}\left(2e(\lambda_{n})\right)}{\Pi^{2}_{n}\left(e(\lambda_{n})\right)} 
\frac{e^{2}\left(\lambda_{n}\right)}{k}
\frac{\Pi_{k-1}\left(e\left(2\lambda_{n}\right)\right)}{\Pi_{k}\left(2e\left(\lambda_{n}\right)\right)}
=\frac{1}{2}\frac{1}{k}\frac{\Pi_{k-1}(1)}{\Pi_{k}\left(2\right)}.
\]
Now using Theorem \ref{Martingale} and the dominated convergence theorem, we get 
\begin{eqnarray*}
\displaystyle \lim_{n\to \infty}\frac{\Pi_{n}\left(2e\left(\lambda_{n}\right)\right)}
{\Pi^{2}_{n}\left(e\left(\lambda_{n}\right)\right)} \displaystyle \sum_{k=1}^{n}\frac{e^{2}\left(\lambda_{n}\right)}{k}
\frac{\Pi_{k-1}\left(e\left(2\lambda_{n}\right)\right)}{\Pi_{k}\left(2e\left(\lambda_{n}\right)\right)}=\frac{1}{2}
\displaystyle \sum_{k=1}^{\infty}\frac{2}{(k+2)(k+1)}=\frac{1}{2}.
\end{eqnarray*}
Therefore, from (\ref{Eq:Martingale}) we obtain 
\begin{equation}
\mathbb{E}\left[\overline{M}_{n}^{2}\left(\lambda_{n}\right)\right]\longrightarrow 1 \mbox{ as } n \to \infty.
\end{equation}
Observing that $\bE\left[ \overline{M}_{n}\left(\lambda_{n}\right) \right] = 1$, we get 
\begin{equation}
\mbox{Var}\left( \overline{M}_{n}\left(\lambda_{n}\right) \right) \rightarrow 0,
\end{equation}
as $n \to \infty$. This implies
\[
\overline{M}_{n}\left(\lambda_{n}\right)\stackrel{p}{\longrightarrow} 1 \mbox{ as } n \to \infty,
\] 
completing the proof of the lemma.
\end{proof}

\begin{proof}[Proof of Theorem \ref{ASd}]
Note that $\Lambda_{n}$ is the random probability measure on $\mathbb{R}^{d}$ corresponding to the random probability vector 
$\frac{1}{n+1}U_{n}$. For $\lambda \in \mathbb{R}^{d}$ the corresponding moment generating function is given by
\begin{equation}
\frac{1}{n+1} \sum_{v \in \mathbb{Z}^{d}}e^{\langle \lambda, v\rangle}U_{n,v}
=
\frac{1}{n+1} U_{n}x\left(\lambda\right)=\frac{1}{n+1}
\overline{M}_{n}\left(\lambda\right)\Pi_{n}\left(e(\lambda)\right).
\end{equation}
The moment generating function corresponding to the scaled and centered random measure $\Lambda^{cs}_{n}$ is 
\begin{eqnarray*}
&   & \frac{1}{n+1}e^{-\langle \lambda, \bmu\sqrt{\log n}\rangle}U_{n}x\left(\frac{\lambda}{\sqrt{\log n}}\right) \\
& = & \frac{1}{n+1}e^{-\langle \lambda, \bmu\sqrt{\log n}\rangle}
\overline{M}_{n}\left(\frac{\lambda}{\sqrt{\log n}}\right)\Pi_{n}\left(e(\frac{\lambda}{\sqrt{\log n}})\right)
\end{eqnarray*} 

To show (\ref{Eq:PrCgs}) it is enough to show that for every subsequence $\{n_{k}\}_{k\geq 1}$, 
there exists a further subsequence $\{n_{k_{j}}\}_{j=1}^{\infty}$ such that as $j\to \infty$ 
\begin{equation}
\label{MGFCS}
\frac{e^{-\langle\lambda,\bmu\sqrt{\log n_{k_{j}}}\rangle}}{n_{k_{j}}+1}
\overline{M}_{n_{k_{j}}}\left(\frac{\lambda}{\sqrt{\log n_{k_{j}}}}\right)
\Pi_{n}\left(e\left(\frac{\lambda}{\sqrt{\log n_{k_{j}}}}\right)\right)\longrightarrow 
e^{\frac{\lambda\varSigma \lambda^{T}}{2}}
\end{equation}
for all $\lambda \in \left[-\delta, \delta\right]^{d}$ almost surely, where $\delta$ is as in 
Theorem \ref{Martingale}.
From Theorem \ref{GRW} we know that 
\[
\frac{Z_{n}-\bmu \log n}{\sqrt{\log n}} \Rightarrow N_{d}\left(0,\mathbb{I}_{d}\right).
\]
Therefore using (\ref{MGF}) as $n \to \infty$ we obtain,
\begin{eqnarray*}
e^{-\langle\lambda,\bmu\sqrt{\log n}\rangle}\mathbb{E}\left[e^{\langle\lambda,\frac{Z_{n}}{\sqrt{\log n}}\rangle}\right]=\frac{1}{n+1}e^{-\langle\lambda,\bmu\sqrt{\log n}\rangle}\Pi_{n}\left(e\left(\frac{\lambda}{\sqrt{ \log n}}\right)\right)\longrightarrow e^{\frac{\lambda\varSigma \lambda^{T}}{2}}. 
\end{eqnarray*}
 
Now using Theorem \ref{Rational2} from the appendix 
it is enough to show (\ref{MGFCS}) only for $\lambda \in \mathbb{Q}^{d}\cap\left[-\delta, \delta\right]^{d}$ 
which is equivalent to proving that for every $\lambda \in \mathbb{Q}^{d}\cap \left[-\delta, \delta\right]^{d}$ 
as $j \to \infty$ 
\begin{eqnarray*}
 \overline{M}_{n_{k_{j}}}\left(\frac{\lambda}{\sqrt{\log n_{k_{j}}}}\right)\longrightarrow 1 \mbox{ almost surely.}
\end{eqnarray*} From Lemma \ref{PC} we know that for all $\lambda \in \left[-\delta, \delta \right]^{d}$
\begin{eqnarray*}
 \overline{M}_{n}\left(\frac{\lambda}{\sqrt{\log n}}\right)\stackrel{p}{\longrightarrow} 1\mbox{ as }n \to \infty.
\end{eqnarray*}
Therefore using the standard diagonalization argument we can say that given a subsequence $\{n_{k}\}_{k\geq1}$ there exists a further subsequence $\{n_{k_{j}}\}_{j=1}^{\infty}$ such that 
for every $\lambda \in \mathbb{Q}^{d}\cap\left[-\delta,\delta \right]^{d}$
\begin{eqnarray*}
 \overline{M}_{n_{k_{j}}}\left(\frac{\lambda}{\sqrt{\log n_{k_{j}}}}\right)\longrightarrow 1 \mbox{ almost surely.}
\end{eqnarray*}
This completes the proof.
\end{proof}

\noindent
{\bf Remark:}
It is worth noting that the proofs of Theorems \ref{GRW} and \ref{ASd} go through if we assume $U_{0}$ to be 
non random probability vector such that there exists $r > 0$ with  
$\sum_{v \in \mathbb{Z}^{d}}e^{\langle\lambda,v\rangle } U_{0,v} < \infty$
whenever $\| \lambda \| < r$.

\subsection{Proofs of the Local Limit Type Results}
In this section, we present the proofs for the local limit theorems. 
As before, we present the proof for $d = 1$ first.

\subsubsection{Proof for the Local Limit Theorems for d=1}
\begin{proof}[Proof of Theorem \ref{llt1}]
Without loss of generality we may assume $\mu=0$ and $\sigma =1$. $X_{j}$ is a lattice random variable, therefore
$I_{j}X_{j}$ is also so. Now by our assumption that $\Pbold\left(X_1 = 0 \right) > 0$, we have 
$0 \in B$, therefore $I_{j}X_{j}$ and $X_{j}$ have the 
same lattice structure. Therefore $Z_{n}$ is a lattice random
variable with lattice $\mathcal{L}_{n}^{(1)}$. Applying Fourier inversion formula, for all 
$x \in \mathcal{L}_{n}^{(1)}$ we obtain 
\begin{eqnarray}
\label{FI1}
\mathbb{P}\left(\frac{Z_{n}}{\sqrt{\log n}}=x\right)
& = &  
\frac{h}{2\pi \sqrt{\log n}}
\int\limits_{-\frac{\pi \sqrt{\log n}}{h}}^{\frac{\pi \sqrt{\log n}}{h}} \! e^{-i tx}\psi_{n}(t)\, \mathrm{d}t \\
& = & 
\frac{1}{2\pi \sqrt{\log n}}
\int\limits_{-\pi \sqrt{\log n}}^{\pi \sqrt{\log n}} \! e^{-i \frac{tx}{h}} \psi_{n}\left(\frac{t}{h}\right)\, \mathrm{d}t 
\end{eqnarray} 
where $\psi_{n}\left(t\right)=\mathbb{E}\left[e^{it\frac{Z_{n}}{\sqrt{\log n}}}\right].$ 
Notice that without loss of any generality, we now can assume $h=1$. 
Also by Fourier inversion formula, for all $x \in \mathbb{R}$
\begin{equation}
\label{FI2}
\phi(x)=\frac{1}{2 \pi }\displaystyle\int\limits_{-\infty}^{\infty}\!{e^{-itx}e^{\frac{-t^{2}}{2}}\, \mathrm{d}t}.
\end{equation}
Given $\epsilon>0$, there exists $N$ large enough such that for all $n\geq N$
\begin{eqnarray*}
&      &  \Big{\lvert}\sqrt{ \log n}\mathbb{P}\left(\frac{Z_{n}}{\sqrt{\log n}}=x\right)-\phi(x)\Big{\rvert} \\
& \leq & \int\limits_{-\pi \sqrt{\log n}}^{\pi\sqrt{\log n}} \! 
         \Big{\lvert}\psi_{n}(t)-e^{\frac{-t^{2}}{2}}\Big{\rvert}\, \mathrm{d}t 
         + 2 \int\limits_{\left[-\pi \sqrt{\log n},\pi \sqrt{\log n}\right]^{c}} \! \phi(t) \, \mathrm{d}t \\
& \leq & \int\limits_{-\pi \sqrt{\log n}}^{\pi \sqrt{\log n}} \! 
         \Big{\lvert}\psi_{n}(t)-e^{\frac{-t^{2}}{2}}\Big{\rvert} \, \mathrm{d}t +\epsilon.
\end{eqnarray*}
Given $M>0$, we can write for all $n$ large enough
\begin{eqnarray}
         \int\limits_{-\pi \sqrt{\log n}}^{\pi \sqrt{\log n}} \!
         \Big{\lvert}{\psi_{n}(t)-e^{\frac{-t^{2}}{2}}\Big{\rvert}\, \mathrm{d}t}
& \leq & \int\limits_{-M}^{M}{\!\Big{\lvert}\psi_{n}(t)-e^{\frac{-t^{2}}{2}}\Big{\rvert}\, \mathrm{d}t} 
         + \int\limits_{M}^{\pi \sqrt{\log n}}{\!\Big{\lvert }\psi_{n}(t)\Big{\rvert}\,\mathrm{d}t} \nonumber\\
&      & \quad +2\displaystyle\int\limits_{M}^{\pi \sqrt{\log n}}\!{e^{\frac{-t^{2}}{2}}\,\mathrm{d}t}. \label{L}
\end{eqnarray}
Given $\epsilon>0$, we choose an $M>0$ such that 
\[
\int\limits \limits_{\left[-M,M\right]^{c}}\!{e^{\frac{-t^{2}}{2}}\,\mathrm{d}t} < \epsilon.
\] Therefore,
\begin{eqnarray}\label{normal}
 \displaystyle \int\limits_{M}^{\pi \sqrt{\log n}}\!{e^{-\frac{t^{2}}{2}}\, \mathrm{d}t}\leq \int\limits \limits_{\left[-M,M\right]^{c}}\!{e^{\frac{-t^{2}}{2}}\,\mathrm{d}t} < \epsilon .
\end{eqnarray} 
 We know from Theorem \ref{GRW} that as $n \to \infty$, $\frac{Z_{n}}{\sqrt{\log n}}\Rightarrow N(0,1)$. Hence
 for all $t \in \mathbb{R}$, $\psi_{n}(t)\longrightarrow 
 e^{\frac{-t^{2}}{2}}$. Therefore, for the chosen $M>0$, by bounded convergence theorem we get as $n \to \infty$
 \begin{eqnarray*}\label{BC}
 \displaystyle\int\limits_{-M}^{M}\!{\Big{\lvert} \psi_{n}(t)-e^{\frac{-t^{2}}{2}}\Big{\rvert}\,\mathrm{d}t}\longrightarrow 0.
 \end{eqnarray*}
 Let
 \begin{eqnarray*}
 \mathcal{I}(n)=\displaystyle\int\limits_{M}^{\pi \sqrt{\log n}}\!{\Big{\lvert} \psi_{n}(t)\Big{\rvert}\,\mathrm{d}t}.
 \end{eqnarray*} 
We will show that as $n \to \infty$, $\mathcal{I}(n)\longrightarrow 0$.
Since 
$Z_{n}\stackrel{d}{=}\textstyle\sum_{j=1}^{n}I_{j}X_{j}$, therefore
\begin{align}
\nonumber \mathbb{E}\left[e^{itZ_{n}}\right]&=\displaystyle \prod_{j=1}^{n}\left(1-\frac{1}{j+1}+\frac{e\left(i 
t\right)}{j+1}\right)\\
\nonumber &=\frac{1}{n+1}\Pi_{n}\left(e\left(it\right)\right)
\end{align} where $e\left(it\right)=\mathbb{E}\left[e^{itX_{1}}\right]$. Therefore,
\[\psi_{n}(t)= \mathbb{E}\left[e^{it\frac{Z_{n}}{\sqrt{\log n}}}\right]=\frac{1}{n+1}\Pi_{n}\left(e(it/\sqrt{\log n})\right).\]
 Applying the change of variables $\frac{t}{\sqrt{\log n}}=w$, we obtain 
 \begin{eqnarray}\label{I}
  \mathcal{I}(n)=\sqrt{\log n}\displaystyle \int\limits_{M/\sqrt{\log n}}^{\pi}\!{\Big{\lvert}\psi_{n}\left(w\sqrt{\log n}\right)\Big{\rvert}\,\mathrm{d}w}.
 \end{eqnarray}
Now there exists $\delta>0$, such that for all $t \in \left(0,\delta\right)$ (see pages 133 of \cite{Durr10})  
\begin{eqnarray}\label{Delta}
 \lvert e\left(it\right)\rvert\leq 1-\frac{t^{2}}{4}.  
\end{eqnarray}
Therefore using the inequality $1-x\leq e^{-x}$, we obtain 
$1-\frac{1}{j+1}+\frac{\lvert e\left(it\right)\rvert}{j+1}\leq e^{-\frac{1}{j+1}\frac{t^{2}}{4}}$.
Hence, for all $t \in \left(0,\delta\right)$
\begin{eqnarray}\label{exp}
 \frac{1}{n+1}\lvert\Pi_{n}\left(e\left(it\right)\right)\rvert \leq e^{-\frac{t^{2}}{4}\displaystyle 
\sum_{j=1}^{n}\frac{1}{j+1}}.
\end{eqnarray} 
We observe from (\ref{I}) that we can write 
\begin{eqnarray*} 
\mathcal{I}(n)=\sqrt{\log n}\displaystyle\int\limits_{M/\sqrt{\log n}}^{\delta}\!{\Big{\lvert}\psi_{n}\left(w\sqrt{\log n}\right)\Big{\rvert}\,\mathrm{d}w}+\sqrt{\log n}\displaystyle\int\limits_{\delta}^{\pi}\!{\Big{\lvert}\psi_{n}\left(w\sqrt{\log n}\right)\Big{\rvert}\,
\mathrm{d}w}.
\end{eqnarray*}Let us write 
\begin{eqnarray*} 
 \mathcal{I}_1(n)=\sqrt{\log n}\displaystyle\int\limits_{M/\sqrt{\log n}}^{\delta}\!{\Big{\lvert}\psi_{n}\left(w\sqrt{\log n}\right)\Big{\rvert} \mathrm{d}w}
\end{eqnarray*}
and 
\begin{eqnarray*}
 \mathcal{I}_{2}(n)=\sqrt{\log n}\displaystyle\int\limits_{\delta}^{\pi}\!{\left\lvert\psi_{n}\left(w\sqrt{\log n}\right)\right\rvert\,\mathrm{d}w}.
\end{eqnarray*}
From (\ref{exp}) we have 
$ \mathcal{I}_1(n)\longrightarrow 0, \mbox{ as } n \to \infty.$
  
Since we have assumed $h=1$ so for all $t \in \left[\delta,2\pi\right)$, 
$\lvert e\left(it\right)\rvert<1$. The characteristic function being continuous in $t$, there exists $0<\eta<1$
such that $\lvert e\left(it\right)\rvert\leq \eta$ for all $t\in \left[\delta, \pi\right]$.
Therefore 
\begin{eqnarray*}
 1-\frac{1}{j+1}+\frac{\lvert e\left(it\right)\rvert}{j+1}\leq 1-\frac{1}{j+1}+\frac{\eta }{j+1}\leq e^{-\frac{1-\eta}{j+1}}.
\end{eqnarray*}
It follows that 
\begin{eqnarray*}
 \frac{1}{n+1}\lvert\Pi_{n}\left(e\left(it\right)\right)\rvert\leq e^{-\displaystyle \sum_{j=1}^{n}\frac{1-\eta}{j+1}}
\leq C_{2}e^{-\left(1-\eta\right)\log n}
\end{eqnarray*} where $C_{2}$ is some positive constant.
So as $n \to \infty$
\begin{eqnarray*}
 \mathcal{I}_2(n)\leq C_{2} e^{-\left(1-\eta\right)\log n}\left(\pi-\delta \right)\sqrt{\log n}\longrightarrow 0. 
\end{eqnarray*}
Combining the facts that 
$\mathcal{I}_1(n)\longrightarrow 0,\mbox{ }\mathcal{I}_2(n)\longrightarrow 0\mbox{ as }n \to \infty$ 
and from (\ref{L}) and
(\ref{normal}), the proof is complete.
\end{proof}

Next we prove Theorem \ref{llt4}. 
\begin{proof}[Proof of Theorem \ref{llt4}]
In this case $\Pbold\left(X_1 = 1 \right) = \Pbold\left(X_1 = -1\right) = \frac{1}{2}$. Thus
the span of $X_1$ is $2$. 
The random variables $I_1 X_1$ is supported on the set $\left\{0,1,-1\right\}$ and it has span $1$.  
We have $\mu =0$ and $\sigma=1$, so from equation \ref{Equ:Def-L^1} we get 
$\mathcal{L}_{n}^{(1)}=\frac{1}{\sqrt{\log n}}\mathbb{Z}$.

For all $x \in \mathcal{L}_{n}^{(1)}$, we obtain by Fourier Inversion formula, 
 \begin{eqnarray*}
 \mathbb{P}\left(\frac{Z_{n}}{\sqrt{\log n}}=x\right)=\frac{1}{2\pi \sqrt{\log n}}\displaystyle
 \int\limits_{-\pi \sqrt{\log n}}^{\pi \sqrt{\log n}}\!{e^{-i tx}\psi_{n}(t)\, \mathrm{d}t} 
 \end{eqnarray*} where $\psi_{n}\left(t\right)=\mathbb{E}\left[e^{it\frac{Z_{n}}{\sqrt{\log n}}}\right].$
Furthermore, by Fourier inversion formula, for all $x \in \mathbb{R}$
\begin{eqnarray*}
\phi(x)=\frac{1}{2 \pi }\displaystyle\int\limits_{-\infty}^{\infty}\!{e^{-itx}e^{\frac{-t^{2}}{2}}\, \mathrm{d}t}.
\end{eqnarray*}
The proof of this theorem is also very similar to that of Theorem \ref{llt1}.
	The bounds for $\Big{\lvert}  \sqrt{ \log n}\mathbb{P}\left(\frac{Z_{n}}{\sqrt{\log n}}=x\right)-\phi(x) \Big{\rvert }$ are similar to that in the proof of
Theorem \ref{llt1} except for that of $\mathcal{I}_2(n)$ 
where 
\begin{eqnarray*}
 \mathcal{I}_2(n)=\sqrt{\log n}\displaystyle\int\limits_{\delta}^{\pi}\!{\left\lvert\psi_{n}\left(w\sqrt{\log n}\right)\right\rvert\,\mathrm{d}w}
\end{eqnarray*} and $\delta $ is chosen as in (\ref{Delta}).
To show that $\mathcal{I}_2(n)\longrightarrow 0 \mbox{ as } n \to \infty$, we observe that 
\begin{align}
\nonumber \mathbb{E}\left[e^{itZ_{n}}\right]&=\displaystyle \prod_{j=1}^{n}\left(1-\frac{1}{j+1}+\frac{\cos t 
}{j+1}\right)\\
\nonumber &=\frac{1}{n+1}\Pi_{n}\left(\cos t\right)
\end{align} since $\mathbb{E}\left[e^{itX_{1}}\right]=\cos t$. Therefore,
\[\psi_{n}(w\sqrt{\log n})= \mathbb{E}\left[e^{iwZ_{n}}\right]=\frac{1}{n+1}\Pi_{n}\left(\cos w\right).\]
We note that $\cos w$ is decreasing in $\left[\frac{\pi}{2},\pi\right]$ and for all $w \in \left[\frac{\pi}{2}, \pi\right] \mbox{ },-1\leq \cos w \leq 0$. Therefore, there exists $\eta >0$( small enough) such that $\left[\pi-\eta, \pi\right)\subset \left(\frac{\pi}{2}, \pi\right] $ and for all $w \in \left[\pi-\eta, \pi\right)$ we have
$-1< \cos(\pi-\eta)<0$ and  
\begin{eqnarray*}
\Big{\lvert}\psi_{n}(w\sqrt{\log n})\Big{\rvert}\leq \frac{1}{n+1}\Pi_{n}\left(\cos (\pi-\eta)\right).
\end{eqnarray*}
Since $-1<\cos (\pi-\eta)<0$, so for all $j\geq1,$ $\left(1+\frac{\cos (\pi -\eta)}{j}\right)<1$. Therefore,
\begin{eqnarray}\label{J2}
\Pi_{n}\left(\cos (\pi-\eta)\right)\leq 1.
\end{eqnarray}
 Let us write 
\[\mathcal{I}_2(n)=\mathcal{J}_{1}(n)+\mathcal{J}_2(n)
\] where 
\begin{eqnarray}\label{J1}
\mathcal{J}_1(n)= \sqrt{\log n}\displaystyle\int\limits_{\delta}^{\pi-\eta}\!{\left\lvert\psi_{n}\left(w\sqrt{\log n}\right)\right\rvert\,\mathrm{d}w}
\end{eqnarray} and 
\begin{eqnarray*}
\mathcal{J}_2(n)= \sqrt{\log n}\displaystyle\int\limits_{\pi -\eta}^{\pi}\!{\left\lvert\psi_{n}\left(w\sqrt{\log n}\right)\right\rvert\,\mathrm{d}w}.
\end{eqnarray*}
It is easy to see from (\ref{J2}) that
\begin{eqnarray*}
\mathcal{J}_2(n)\leq \frac{\eta}{n+1}\sqrt{\log n}\longrightarrow 0 \mbox{ as } n \to \infty.
\end{eqnarray*}
For all $t\in \left[\delta, \pi-\eta \right],\mbox{ } 0\leq \lvert \cos t\rvert <1$, so there exists $0<\alpha<1$ such that 
$0\leq \lvert \cos t \rvert \leq \alpha$ for all $t\in \left[\delta, \pi-\eta \right]$. Recall that 
\[\psi_{n}(w\sqrt{\log n})=\prod_{j=1}^{n}\left(1-\frac{1}{j+1}+\frac{\cos w}{j+1}\right). 
\] Using the inequality $1-x\leq e^{-x}$, it follows that for all $t\in \left[\delta, \pi-\eta \right]$
\begin{eqnarray*}
1-\frac{1}{j+1}+\frac{\lvert \cos t\rvert}{j+1}\leq 1-\frac{1}{j+1}+\frac{\alpha }{j+1}\leq e^{-\frac{1-\alpha}{j+1}}
\end{eqnarray*}
and hence
\begin{eqnarray*}
 \frac{1}{n+1}\lvert\Pi_{n}\left(\cos t\right)\rvert\leq e^{-\displaystyle \sum_{j=1}^{n}\frac{1-\alpha}{j+1}}
\leq Ce^{-\left(1-\eta\right)\log n}
\end{eqnarray*} where $C$ is some positive constant. Therefore from (\ref{J1}) we obtain as $n \to \infty$
\begin{eqnarray*}
 \mathcal{J}_1(n)\leq C e^{-\left(1-\alpha \right)\log n}\left(\pi-\eta -\delta \right)\sqrt{\log n}\longrightarrow 0. 
\end{eqnarray*}
\end{proof}

\subsubsection{Proofs for the Local Limit Type Results for $d\geq2$}
\begin{proof}[Proof of Theorem \ref{llt2} ]
 Without loss of generality we may assume that $\bmu=0$ and $\varSigma=\mathbb{I}_{d}$. 
$X_{j}$ being a lattice random variable, $I_{j}X_{j}$ is also so. 
By our assumption $\Pbold\left(X_1 = 0 \right) > 0$, so 
$0 \in B$, therefore
$X_{j}$ and $I_{j}X_{j}$ are supported on the same lattice.
For $A\subset \mathbb{R}^{d}$ and $x\in \mathbb{R}$, we define
$xA=\{xy\colon y \in A\}$.
By Fourier inversion formula (see 21.28 on page 230 of \cite{BhRa76}), we get for $x\in \mathcal{L}_{n}^{(d)}$
\begin{eqnarray*}
 \mathbb{P}\left(\frac{Z_{n}}{\sqrt{\log n}}=x\right)=\frac{l}{(2\pi\sqrt{\log n})^{d}}\int\limits\limits_{(\sqrt{\log n}\mathcal{F}^{*})}
\!{\psi_{n}(t)e^{-i\langle t,x\rangle}\, \mathrm{d}t} 
\end{eqnarray*} where $\psi_{n}(t)=\mathbb{E}\left[e^{i\langle t,\frac{Z_{n}}{\sqrt{\log n}}\rangle}\right]$, $l=\lvert\text{det}\left(\mathcal{L}\right)\rvert$ and
$\mathcal{F}^{*}$ is the fundamental domain for $X_{1}$ as defined in equation(21.22) on page 229 of \cite{BhRa76}.
Also by Fourier inversion formula
\begin{eqnarray*}
 \phi_{d}(x)=\frac{1}{(2\pi)^{d}}\int\limits\limits_{\mathbb{R}^{d}}\!
 {e^{-i\langle t,x\rangle}e^{-\frac{\|t\|^{2}}{2}}\,\mathrm{d}t}.
\end{eqnarray*}
Given $\epsilon>0$, there exists $N>0$ such that $n\geq N$,
\begin{eqnarray*}
&  & \Big{\lvert} \frac{\left(\sqrt{\log n} \right)^{d}}{l}
\mathbb{P}\left(\frac{Z_{n}}{\sqrt{\log n}}=x\right)-
\phi_{d}(x)\Big{\rvert } \\
&\leq & \frac{1}{(2\pi)^{d}}\int\limits\limits_{(\sqrt{\log n}\mathcal{F}^{*})}\!
\Big{\lvert}\psi_{n}(t)-
e^{-\frac{\|t\|^{2}}{2}}\Big{\rvert} \, \mathrm{d}t
+\frac{1}{(2\pi)^{d}}\int\limits \limits_{\mathbb{R}^{d}\setminus \sqrt{\log n}\mathcal{F}^{*} }\!{e^{-\frac{\|t\|^{2}}{2}}\,\mathrm{d}t}\\
&\leq & \frac{1}{(2\pi)^{d}}\int\limits\limits_{(\sqrt{\log n}\mathcal{F}^{*})}\!{\Big{\lvert}\psi_{n}(t)-
e^{-\frac{\|t\|^{2}}{2}}\Big{\rvert}\, \mathrm{d}t}+ \epsilon.
\end{eqnarray*}
Given any compact set $A\subset\mathbb{R}^{d}$ for all $n$ large enough 
\begin{eqnarray*}
         \int\limits\limits_{(\sqrt{\log n}\mathcal{F}^{*})} \! 
         {\Big{\lvert}\psi_{n}(t)-e^{-\frac{\|t\|^{2}}{2}}\Big{\rvert}\, \mathrm{d}t}
& \leq & \int\limits\limits_{A} \! \Big{\lvert}\psi_{n}(t)- e^{-\frac{\|t\|^{2}}{2}}\Big{\rvert}\, \mathrm{d}t
         + \int\limits_{(\sqrt{\log n}\mathcal{F}^{*})\setminus A} \!
         \Big{\lvert}\psi_{n}(t)\Big{\rvert}\, \mathrm{d}t \\
&      & \qquad \qquad + \int\limits_{\mathbb{R}^{d}\setminus A}\! e^{-\frac{\|t\|^{2}}{2}}\, \mathrm{d}t.
\end{eqnarray*}
By Theorem \ref{GRW}, we know that $\frac{Z_{n}}{\sqrt{\log n}}\Rightarrow N_{d}(0,\mathbb{I}_{d})$ as $n \to \infty$. Therefore,
for any compact set $A\subset \mathbb{R}^{d}$ by bounded convergence theorem,
\begin{eqnarray*}
 \int\limits\limits_{A}\!{\Big{\lvert} \psi_{n}(t)-e^{-\frac{\|
t\|^{2}}{2}}\Big{\rvert}\mathrm {d}t} \longrightarrow 0 \mbox{ as } n \to \infty.
\end{eqnarray*} 
Choose $A$ such that 
\begin{eqnarray*}
 \int\limits\limits_{A^{c}}\!{e^{-\frac{\|t\|^{2}}{2}}\,\mathrm{d}t}<\epsilon.
\end{eqnarray*}
Let us write
\begin{eqnarray}\label{Int}
 \mathcal{I}(n)=\int\limits\limits _{(\sqrt{\log n}\mathcal{F}^{*})\setminus A}\!{\Big{\lvert}\psi_{n}(t)\Big{\rvert}\,\mathrm{d}t}.
\end{eqnarray}
For the above choice of $A$, we will show that 
\begin{eqnarray*}
\mathcal{I}(n)\longrightarrow 0 \mbox{ as } n \to \infty.
\end{eqnarray*}
Since $Z_{n}\stackrel{d}{=}\textstyle \sum_{j=1}^{n}I_{j}X_{j}$, we have
\begin{align}
\nonumber \mathbb{E}\left[e^{i\langle t,Z_{n}\rangle}\right]&=\displaystyle \prod_{j=1}^{n}\left(1-\frac{1}{j+1}+\frac{e\left(i 
t\right)}{j+1}\right)\\
\nonumber &=\frac{1}{n+1}\Pi_{n}\left(e\left(it\right)\right)
\end{align} where $e\left(it\right)=\mathbb{E}\left[e^{i\langle t,X_{1}\rangle}\right]$. So,
\[\psi_{n}(t)=\mathbb{E}\left[e^{i\langle t,\frac{Z_{n}}{\sqrt{\log n}}\rangle}\right]=\frac{1}{n+1}\Pi_{n}\left(e\left(\frac{1}{\sqrt{\log n}}it\right)\right).\]
Applying the change of variables $t=\frac{1}{\sqrt{\log n}}w$ to (\ref{Int}), we obtain 
\begin{eqnarray}\label{CV}
 \mathcal{I}(n)=(\sqrt{\log n})^{d}\int\limits\limits_{\mathcal{F}^{*}\setminus \frac{1}{\sqrt{\log n}}A}\!{\Big{\lvert}\psi_{n}\left(\sqrt{\log n}w\right)\Big{\rvert}\,\mathrm{d}w}.
\end{eqnarray}
We can choose a $\delta>0$, such that for all $w\in B(0,\delta)\setminus \{0\}$ there exists $b>0$ such that 
\begin{eqnarray}\label{RWB}
 \lvert e(iw)\rvert \leq 1-\frac{b\|w\|^{2}}{2},
\end{eqnarray}
(see Lemma 2.3.2(a) of \cite{LaLi10} for a proof).
Therefore, using the inequality $1-x\leq e^{-x}$ we have
\begin{eqnarray}
 \nonumber\lvert\psi_{n}(\sqrt{\log n}w)\rvert &=& \frac{1}{n+1}\lvert\Pi_{n}\left(e(iw)\right)\rvert\\
 \nonumber &\leq& \displaystyle \prod_{j=1}^{n+1}\left(1-\frac{1}{j+1}+\frac{\lvert e(iw)\rvert}{j+1}\right)\\
\label{Dbound} &\leq & e^{-\displaystyle\sum_{j=1}^{n}\frac{b}{j+1}\frac{\|w\|^{2}}{2}} \leq C_{1}e^{-b  \frac{\|w\|^{2}}{2}\log n}
\end{eqnarray} for some positive constant $C_{1}$. 
From (\ref{CV}) we can write
\begin{eqnarray*}
 \mathcal{I}(n)=\mathcal{I}_1(n)+ \mathcal{I}_2(n)
\end{eqnarray*}
 where 
\begin{eqnarray*} 
 \mathcal{I}_1(n)=(\sqrt{\log n})^{d}\int\limits\limits_{\left(B(0,\delta)\setminus \frac{1}{\sqrt{\log n}}A\right)\cap \mathcal{F}^{*}}\!{\lvert{\psi_{n}
\left(\sqrt{\log n}w\right)\rvert}\,\mathrm{d}w}
\end{eqnarray*} and
\begin{eqnarray*}
 \mathcal{I}_2(n)=(\sqrt{\log n})^{d}\int\limits_{\mathcal{F}^{*}\setminus B(0,\delta)}{\lvert\psi_{n}\left(\sqrt{\log n}w\right)\rvert dw}.
\end{eqnarray*}
Since (\ref{Dbound}) holds, given $\epsilon>0$, we have for all $n$ large enough
\begin{eqnarray}\label{I1}
 \mathcal{I}_1(n)\leq (\sqrt{\log n})^{d}\int\limits\limits_{B(0,\delta)\setminus \frac{A}{\sqrt{\log n}}}\!{C_{1}e^{-b\frac{\|w\|^{2}}{2}\log n}\,\mathrm{d}w}\leq \epsilon.
\end{eqnarray} 
 Since the lattices for $X_1$ and $I_1 X_1$ are same, for all $w \in \mathcal{F}^{*}\setminus B(0,\delta)$, we get \\
$\lvert e(iw)\rvert<1$, so there exists an $0<\eta<1$, such that 
$\lvert e(iw)\rvert \leq \eta$. Therefore, using the inequality $1-x\leq e^{-x}$, we obtain
\begin{eqnarray}\label{Exp1}
 \lvert\psi_{n}(\sqrt{\log n}w)\rvert\leq e^{- \sum_{j=i}^{n}\frac{1}{j+1}(1-\eta)}\leq C_{2}e^{-(1-\eta)\log n}
\end{eqnarray} for some positive constant $C_{2}$.
Therefore, using equation (21.25) on page 230 of \cite{BhRa76} we obtain 
\begin{eqnarray*}
 \mathcal{I}_2(n)\leq C'_{2}(\sqrt{\log n})^{d}e^{-(1-\eta)\log n}\longrightarrow 0 \mbox{ as } n \to \infty
\end{eqnarray*} where $C'_{2}$ is an appropriate positive constant.
\end{proof}

\begin{proof}[Proof of the Theorem \ref{llt5}]
In this case $\Pbold\left(X_1 = \pm e_i\right) = \frac{1}{2d}$ for $1 \leq i \leq d$, where $e_i$ is the
$i$-th unit vector in direction $i$, thus $\mu = 0$ and $\varSigma = \frac{1}{d} {\mathbb I}_d$.

For notional simplicity we consider the case $d = 2$, the general case can be written similarly. 

Now for each $j \in \mathbb{N}$, $I_{j}X_{j}$ is a lattice random vector with the minimal lattice $\mathbb{Z}^{2}$. It is easy to note that 
$2\pi \mathbb{Z}\times 2\pi\mathbb{Z}$ is the set of all periods for $I_{j}X_{j}$ and its fundamental domain is given by $\left(-\pi,\pi\right)^{2}$. 
To prove (\ref{lltSSRW}), it is enough to show
\begin{eqnarray*}
\displaystyle \sup_{x \in \frac{1}{\sqrt{2}}\mathcal{L}_{n}^{(2)}} \left\vert \left(\log n\right)
\mathbb{P}\left(\frac{Z_{n}}{\sqrt{\log n}}=x\right)-
\phi_{2,\frac{1}{2}\mathbb{I}_{2}}(x) \right\vert \longrightarrow 0 \mbox{ as } n \to \infty,
\end{eqnarray*} where $\phi_{2,\frac{1}{2}\mathbb{I}_{2}}(x)=\frac{1}{\pi}e^{-\|x\|^2}$ is the
bivariate normal density with mean vector $0$ and variance-covariance matrix $\frac{1}{2}\mathbb{I}_{2}$
and 
$\frac{1}{\sqrt{2}}\mathcal{L}_{n}^{(2)}=\frac{1}{{\sqrt{\log n}}}\mathbb{Z}^{2}$.
By Fourier inversion formula (see 21.28 on page 230 of \cite{BhRa76}), 
we get for $x\in \frac{1}{\sqrt{2}}\mathcal{L}_{n}^{(2)}$,
\begin{eqnarray*}
 \mathbb{P}\left(\frac{Z_{n}}{\sqrt{\log n}}=x\right)=\frac{1}{(2\pi)^{2} \log n}\int\limits\limits_{\left(-\sqrt{\log n}\pi,\sqrt{\log n}\pi\right)^{2}}
\!{\psi_{n}(t)e^{-i\langle t,x\rangle}\, \mathrm{d}t}. 
\end{eqnarray*} 
Also by Fourier inversion formula
\begin{eqnarray*}
 \phi_{2,\frac{1}{2}\mathbb{I}_{2}}(x)=\frac{1}{(2\pi)^{2}}\int\limits\limits_{\mathbb{R}^{2}}\!{e^{-i\langle t,x\rangle}e^{-\frac{\|t\|^{2}}{4}}\,\mathrm{d}t}.
\end{eqnarray*} Let us write $H_{n}=\left(-\sqrt{\log n}\pi,\sqrt{\log n}\pi\right)^{2}.$
Given $\epsilon>0$, there exists $N>0$ such that $n\geq N$,
\begin{eqnarray*}
\nonumber \Big{\lvert} \log n \,
\mathbb{P}\left(\frac{Z_{n}}{\sqrt{\log n}}\right)-
\phi_{2,\frac{1}{2}\mathbb{I}_{2}}(x)\Big{\rvert }
& \leq & \frac{1}{(2\pi)^{2}}\int\limits\limits_{H_{n}}\!{\Big{\lvert}\psi_{n}(t)-
e^{-\frac{\|t\|^{2}}{4}}\Big{\rvert}\, \mathrm{d}t}\\
&\mbox{ }& \quad + \frac{1}{(2\pi)^{2}}\int\limits \limits_{\mathbb{R}^{2}\setminus H_{n}}\!{e^{-\frac{\|t\|^{2}}{4}}\,\mathrm{d}t}\\
&\leq & \frac{1}{(2\pi)^{2}}\int\limits\limits_{H_{n}}\!{\Big{\lvert}\psi_{n}(t)-
e^{-\frac{\|t\|^{2}}{4}}\Big{\rvert}\, \mathrm{d}t}+ \epsilon.
\end{eqnarray*}
Given any compact set $A\subset\mathbb{R}^{2}$, for all $n$ large enough we have 
\begin{eqnarray}\label{III}
\nonumber\int\limits\limits_{H_{n}}\!{\Big{\lvert}\psi_{n}(t)-
e^{-\frac{\|t\|^{2}}{4}}\Big{\rvert}\, \mathrm{d}t}\leq \int\limits\limits_{A}\!{\Big{\lvert}\psi_{n}(t)- e^{-\frac{\|t\|^{2}}{4}}\Big{\rvert}\, \mathrm{d}t}+ \int\limits_{H_{n}\setminus A}{\Big{\lvert}\psi_{n}(t)\Big{\rvert}\, \mathrm{d}t}+
 \int\limits_{\mathbb{R}^{2}\setminus A}\!{e^{-\frac{\|t\|^{2}}{4}}\, \mathrm{d}t}.
\end{eqnarray}
By Theorem \ref{GRW}, we know that $\frac{Z_{n}}{\sqrt{\log n}}\Rightarrow N_{2}(0,2^{-1}\mathbb{I}_{2})$ as $n \to \infty$. Therefore,
for any compact set $A\subset \mathbb{R}^{2}$ by bounded convergence theorem,
\begin{eqnarray*}
 \int\limits\limits_{A}\!{\Big{\lvert} \psi_{n}(t)-e^{-\frac{\|
t\|^{2}}{4}}\Big{\rvert}\mathrm {d}t} \longrightarrow 0 \mbox{ as } n \to \infty.
\end{eqnarray*} 
Choose $A$ such that 
\begin{eqnarray*}
 \int\limits\limits_{A^{c}}\!{e^{-\frac{\|t\|^{2}}{4}}\,\mathrm{d}t}<\epsilon.
\end{eqnarray*}
Let us write
\begin{eqnarray*}
 \mathcal{I}(n)=\int\limits\limits _{H_{n}\setminus A}\!{\Big{\lvert}\psi_{n}(t)\Big{\rvert}\,\mathrm{d}t}.
\end{eqnarray*}
For the above choice of $A$, we will show that 
\begin{eqnarray*}
\mathcal{I}(n)\longrightarrow 0 \mbox{ as } n \to \infty.
\end{eqnarray*}
Applying the change of variables $t=\frac{1}{\sqrt{\log n}}w$, we obtain 
\begin{eqnarray*}
 \mathcal{I}(n)=\log n\int\limits\limits_{\left(-\pi,\pi\right)^{2}\setminus \frac{1}{\sqrt{\log n}}A}\!{\Big{\lvert}\psi_{n}\left(\sqrt{\log n}w\right)\Big{\rvert}\,\mathrm{d}w}.
\end{eqnarray*} where for $A \subset \mathbb{R}^{d}$ and $x \in \mathbb{R}$, we write $xA=\{xy\colon y \in A\}$.
We can write
\begin{eqnarray*}
 \mathcal{I}(n)=\mathcal{I}_1(n)+ \mathcal{I}_2(n)
\end{eqnarray*}
 where 
\begin{eqnarray*} 
 \mathcal{I}_1(n)=\log n \int\limits\limits_{\left(B(0,\delta)\setminus \frac{1}{\sqrt{\log n}}A\right)\cap \left(-\pi,\pi\right)^{2}}\!{\lvert{\psi_{n}
\left(\sqrt{\log n}w\right)\rvert}\,\mathrm{d}w}
\end{eqnarray*} and
\begin{eqnarray*}
 \mathcal{I}_2(n)=\log n\int\limits\limits_{\left(-\pi,\pi\right)^{2}\setminus B(0,\delta)}{\lvert\psi_{n}\left(\sqrt{\log n}w\right)\rvert dw}.
\end{eqnarray*} where $\delta$ is as in (\ref{RWB}).
Using arguments similar to (\ref{I1}), we can show that $\mathcal{I}_1(n)\longrightarrow 0 \mbox{ as } n \to \infty.$ Therefore it is enough to show that $\mathcal{I}_2(n)\longrightarrow 0 \mbox{ as } n \to \infty.$
To do so, we first observe that for $t=\left(t^{(1)},t^{(2)}\right)\in \mathbb{R}^{2}$ the characteristic function for $X_{1}$ is given by $e\left(it\right)=\frac{1}{2}\left(\cos t^{(1)}+\cos t^{(2)}\right)$. 
If $t\in \left[-\pi,\pi\right]^{2}$  be such that 
$\lvert e\left(it\right)\rvert=1,$ then $t\in \{(\pi,\pi),(-\pi,\pi),(\pi,-\pi),(-\pi,-\pi)\}$.
The function $\cos \theta$ is continuous and decreasing as a function of $\theta $ for $t \in \left[\frac{\pi}{2},\pi\right]$. Choose $\eta >\frac{\pi}{2}$ such that for $t\in A_{1}=(-\pi,\pi)^{2}\cap B^{c}(0,\delta)\cap D^{c}, \mbox{ we have }\lvert e\left(it\right)\rvert<1,$ where $D=\left[\pi-\eta,\pi \right)^{2}\cup\left[-\pi-\eta,-\pi\right)\times \left[\pi-\eta,\pi\right)\cup \left[-\pi-\eta, -\pi\right)^{2}\cup \left[\pi-\eta,\pi\right)\times \left[-\pi-\eta,-\pi\right)$. Let us write 
\[\mathcal{I}_2(n)=\mathcal{J}_1(n)+\mathcal{J}_2(n)
\] where 
\begin{eqnarray*}
\mathcal{J}_1(n)=\log n\int\limits_{A_{1}}{\lvert\psi_{n}\left(\sqrt{\log n}w\right)\rvert dw}
\end{eqnarray*} and
\begin{eqnarray*}
\mathcal{J}_2(n)=\log n\int\limits_{D}{\lvert\psi_{n}\left(\sqrt{\log n}w\right)\rvert dw}\mbox{.}
\end{eqnarray*} It is easy to note that 
\[\mathcal{J}_1(n)\leq \log n \int\limits\limits_{\overline{A}_{1}}{\lvert\psi_{n}\left(\sqrt{\log n}w\right)\rvert dw}
\] where $\overline{A}_{1}$ denotes the closure of $A_{1}$. For $w \in \overline{A}_{1}$ there exists some $0<\alpha<1$ such that $\lvert e\left(it\right)\rvert\leq \alpha$. Therefore using bounds similar to that in (\ref{Exp1}) we can show that 
\[\mathcal{J}_1(n)\longrightarrow 0 \mbox{ as } n \to \infty.
\]
We observe that 
\[\mathcal{J}_2(n)\leq 4 \log n\int\limits_{\left[\pi-\eta,\pi\right]^{2}}{\lvert\psi_{n}\left(\sqrt{\log n}w\right)\rvert dw}.\] Hence, it is enough to show that 
$\log n\int\limits_{\left[\pi-\eta,\pi\right]^{2}}{\lvert\psi_{n}\left(\sqrt{\log n}w\right)\rvert dw}\longrightarrow 0$
as $n \to \infty$.   
For $w \in \left[\pi-\eta,\pi \right]^{2}$ we have $0<\lvert \left(1+\frac{e(iw)}{j}\right)\rvert\leq \left(1+\frac{\cos (\pi-\eta)}{j}\right)\leq 1$. Therefore, 
\begin{eqnarray*}
\lvert\psi_{n}(w)\rvert=\displaystyle\frac{1}{n+1}\prod_{j=1}^{n}\Big{\lvert}\left(1+\frac{e(iw)}{j}\right)\Big{\rvert}\leq \frac{1}{n+1}.
\end{eqnarray*} So,
\begin{eqnarray*}
\log n\int\limits_{\left[\pi-\eta,\pi\right]^{2}}{\lvert\psi_{n}\left(\sqrt{\log n}w\right)\rvert dw}\leq \frac{\eta^{2}}{n+1}\log n\longrightarrow 0 \mbox{ as } n \to \infty.
\end{eqnarray*}
\end{proof}

\section{Urns with Colors Indexed by Other Lattices on $\mathbb{R}^{d}$}
\label{Sec:General}
We can further generalize the urn models with color sets indexed by certain countable lattices in $\Rbold^d$.
Such a model will be associated with the corresponding random walk on the lattice. To state the results
rigorously we consider the following notations. 

Let $\{X_{i}\}_{i\geq 1}$ be a sequence of random $d$-dimensional i.i.d. vectors with non empty support 
set $B \subseteq \mathbb{R}^{d}$ and probability mass function $p$. We assume that $B$ is finite. 
Consider the countable subset
\[
S^d := \left\{\textstyle \sum_{i=1}^{k}n_{i}b_{i}\colon n_{1},n_{2},\ldots, n_{k} \in \mathbb{N},
b_{1},b_{2},\ldots,b_{k}\in B\right\}
\]
of $\Rbold^d$ which will index the set of colors. 

Like earlier we consider $S_n := X_0 + X_1 + \cdots + X_n, n \geq 0$, the random walk starting at $X_0$. 
The transition matrix for this work is given by 
\[ R := \left(\left( p\left(v - u\right) \right)\right)_{u, v \in S^d}.\]
We say a process $\left(U_n\right)_{n \geq 0}$ is a urn scheme with colors indexed by $S^d$
and replacement matrix $R$ and starting configuration $U_0$,
if $\left(U_n\right)_{n \geq 1}$ is defined recursively by the equation
\begin{equation}
\label{recurssion-S}
U_{n+1}=U_{n} + \zeta_{n+1} R 
\end{equation}
where $\zeta_{n+1} = \left(\zeta_{n+1,v}\right)_{v \in S^d}$ is such that 
$\zeta_{n+1,V}=1$ and $\zeta_{n+1,u} = 0$ if $u \neq V$ where $V$ is a random color
chosen from the configuration $U_n$. In other words for $n \geq 0$,
\[
U_{n+1}=U_n + R_V
\]
where $R_V$ is the  $V^{\text{th}}$ row of the replacement matrix $R$. 
Following the same nomenclature as done earlier,
we will call this process the
\emph{infinite color urn model associated with the random walk $\left\{S_n\right\}_{n \geq 0}$ on $S^d$}.
Naturally, when $S^d = \Zbold^d$, this process is exactly the one discussed earlier.

We will use same notations as earlier for the mean, non-centered dispersion 
matrix and moment generating function for the increment $X_1$ (see \eqref{Equ:Basic-Notations} 
for the definitions). 
Like earlier we denote by $Z_n$ the $\left(n+1\right)$-th selected color. Just like in the 
previous case, 
the expected proportion of colors in the urn at time $n$ will be given by the distribution of $Z_n$
but now on $S^d$. 

From the proof of Theorem \ref{LRW1} it follows that the result holds also for this generalization. 
This enable us to generalize
Theorem \ref{GRW} and Theorem \ref{ASd} as follows. 
\begin{theorem}
\label{GRWg}
Let $\overline{\Lambda}_{n}$ be the probability measure on $\mathbb{R}^{d}$ corresponding to the probability 
vector $\frac{1}{n+1}\left(\mathbb{E}[U_{n,v}]\right)_{v \in S^d}$ and let
\[
\overline{\Lambda}_{n}^{cs}(A)
:=\overline{\Lambda}_{n}\left(\sqrt{\log n}A\varSigma^{-1/2}+ \bmu\log n\right), \,\,\,
A \in \mathcal{B}\left(\mathbb{R}^{d}\right).
\]
Then, as $n \to \infty$,
\begin{equation}
\overline{\Lambda}_{n}^{cs}\Rightarrow \Phi_{d}.
\end{equation}
\end{theorem}
\begin{theorem}
Let $\Lambda_{n} \in \mathcal{M}_{1}$ be the random probability measure corresponding to the 
random probability vector $\frac{U_{n}}{n+1}$. Let 
\[
\Lambda^{cs}_{n}\left(A\right)
=\Lambda_{n}\left(\sqrt{\log n}A\varSigma^{-1/2}+  \bmu \log n\right).
\]
where $A$ is a Borel subset of $\Rbold^d$. 
Then, as $n \to \infty $,
\begin{equation}
\label{Eq:PrCgsg}
\Lambda_{n}^{cs}\stackrel{p}{\longrightarrow} \Phi_{d} \mbox{\ in\ } \mathcal{M}_{1}.
\end{equation}
\end{theorem} 
The proofs of these two theorems are exactly similar to their counter parts and hence are omitted.

As an application 
we now consider a specific example, namely, the triangular lattice in two dimension. For this
the support set for the i.i.d. increment vectors is given by  
\[
B= \left\{(1,0), (-1,0),\omega, -\omega, \omega^{2}, -\omega^{2} \right\},
\]
where $\omega, \omega^{2}$ are the complex cube roots of unity (see Figure \ref{Fig:Tri}). 
The law of $X_1$ is uniform on $B$. This gives the 
random walk on the triangular lattice in two dimension.
\begin{figure}
\centering
\includegraphics [scale=0.45] {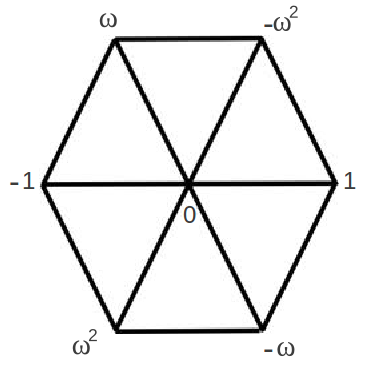}
\caption{Triangular Lattice}
\label{Fig:Tri}
\end{figure}
The following is an immediate corollary of Theorem \ref{GRWg}.
\begin{cor}
\label{TRW}
Consider the urn model associated with the random walk on two dimensional triangular lattice then 
as $n \rightarrow \infty$
\begin{eqnarray}
\frac{Z_{n}}{\sqrt{\log n}} \Rightarrow N_{2}\left(0,\frac{1}{2}\mathbb{I}_{2}\right).
\end{eqnarray}
\end{cor}

\begin{proof}
Since $1+\omega+\omega^{2}=0$, therefore it is immediate that $\bmu=0$. Also we know that 
$\omega=\frac{1}{2}+\mathit{i}\frac{\sqrt{3}}{2}$.
Writing $\omega =\left( \mathit{Re}\mbox{ } \omega + i \mathit{Im}\mbox{ } \omega\right)$, we get
\[ 
\mathbb{E}\left[\left(X_{1}^{(1)}\right)^{2}\right]= 
\frac{2}{6}\left(1+\left(\mathit{Re}\mbox{ }
\omega\right)^{2}+ 
\left(\mathit{Re}\mbox{ }\omega^{2}\right)^{2}\right).
\]
Since $\mathit{Re}\mbox{ }\omega=\mathit{Re}\mbox{ }\omega^{2}$, 
therefore 
\[
\mathbb{E}\left[\left(X_{1}^{(1)}\right)^{2}\right]
=\frac{2}{6}\left(1+2\left(\mathit{Re}\mbox{ }\omega\right)^{2}\right)=\frac{1}{2}.
\]
Similarly, $\mathit{Im}(\omega)=-\mathit{Im}(\omega^{2})$, 
and hence $\mathbb{E}\left[\left(X_{1}^{2}\right)^{2}\right]=
\frac{2}{6}\left(\left(\mathit{Im}(\omega)\right)^{2}+\left(\mathit{Im}(\omega^{2})\right)^{2}\right)=\frac{1}{2}$.
Finally 
\[
\mathbb{E}\left[X_{1}^{(1)}X_{1}^{(2)}\right]=-\frac{2}{6}\mathit{Im}\left(1+\omega+\omega^{2}\right)=0.
\]
So $\varSigma = \frac{1}{2} {\mathbb I}_2$. The rest is just an application of Theorem \ref{GRWg}.
\end{proof}

\section*{Appendix}
We present here an elementary but technical result which we have used in the proof of Theorem \ref{ASd}.
It is really a generalization of the classical result for Laplace transform, namely, 
Theorem 22.2 of \cite{Bi95}.
\begin{theorem}\label{Rational2}
 Let $\nu_{n}$ be a sequence of probability measures on $\left(\mathbb{R}^{d}, 
\mathcal{B}(\mathbb{R}^{d})\right)$ and let $m_{n}(\cdotp)$
 be the corresponding moment generating functions. Suppose there exists $\delta>0$ such that    
$m_{n}(\lambda)\longrightarrow e^{\frac{\|\lambda\|^{2}}{2}}\mbox{ as }
n \to \infty$ for every $\lambda \in \left[-\delta, \delta\right]^{d}\cap \mathbb{Q}^{d}$,
then  as $n \to \infty$ 
\begin{equation}
\nu_{n} \Rightarrow \Phi_{d}.
\label{Equ:nu_n-to-normal}
\end{equation}
\end{theorem}

\begin{proof}
Choose a  $\delta' \in \mathbb{Q}$ such that $0 < \delta' < \delta$, and observe that for every $ a > 0$ 
\[
\nu_{n}\left(\left(\left[-a,a\right]^d\right)^{c}\right) \leq 
\sum_{i=1}^d e^{-\delta' a}\left(m_{n}(-\delta' e_i) + m_{n}(\delta' e_i)\right),
\]
where $\left\{e_i\right\}_{i=1}^d$ are the $d$-unit vectors. 
Now for our assumption we get 
$m_{n}(\delta' e_i) \to e^{\frac{{\delta'}^2}{2}}$ and 
$m_{n}(-\delta' e_i)\to e^{\frac{{\delta'}^2}{2}}$ as $ n \to \infty$
for every $1 \leq i \leq d$. Thus we get
\begin{eqnarray*}
\sup_{n \geq 1} \nu_{n} \left(\left(\left[-a,a\right]^d\right)^{c}\right)
\longrightarrow 0 \mbox{ as } a \to \infty.
\end{eqnarray*}
So the sequence of probability measures $\left(\nu_n\right)_{n \geq 1}$ is tight. 
Therefore, for every subsequence $\{n_{k}\}_{k\geq 1}$ there exists a further subsequence 
$\{n_{k_{j}}\}_{j\geq 1}$ and a probability measure $\nu$ such that as $n \to \infty$,
\[ \nu_{n_{k_{j}}}  \Rightarrow \nu. \] 
Then by dominated convergence theorem 
\[ 
m_{n_{k_j}}\left(\lambda\right) \longrightarrow m_{\infty}\left(\lambda\right), \,\,\, \forall \,\,
\lambda \in \left(- \delta, \delta\right)^d \cap {\mathbb Q}^d 
\]
where $m_{\infty}$ is the moment generating function of $\nu$. 
But from our assumption 
\[ 
m_{n_{k_j}}\left(\lambda\right) \rightarrow e^{\frac{\| \lambda \|^2}{2}}, \,\,\, \forall \,\,
\lambda \in \left[-\delta, \delta\right]^d \cap \mathbb{Q^d}.
\]
So we conclude that 
\[
m_{\infty}\left(\lambda\right) = e^{\frac{\| \lambda \|^2}{2}}, , \,\,\, \forall \,\,
\lambda \in \left(- \delta, \delta\right)^d \cap {\mathbb Q}^d.
\]
Since both sides of the above equation are continuous functions on their respective domains, we get that
$m_{\infty}\left(\lambda\right) = e^{\frac{\| \lambda \|^2}{2}}$ for every
$\lambda \in \left(- \delta, \delta\right)^d$. But the standard Gaussian distribution 
is characterize by the values of 
its moment generating function in a open neighborhood of $0$, so
we conclude that every sub-sequential limit is standard Gaussian. 
This proves \eqref{Equ:nu_n-to-normal}.
\end{proof}

\section*{Acknowledgement}
The authors are grateful to Krishanu Maulik and Codina Cotar for various discussions they had with them. 

\bibliographystyle{plain}

\bibliography{Urn-Model}

\end{document}